\documentclass[reqno,11pt]{amsart}

\usepackage[T1]{fontenc}
\usepackage[latin1]{inputenc}
\usepackage{times}
\usepackage{color}
\usepackage{psfrag,graphicx}
\usepackage{amsfonts,amssymb}
\usepackage{amsmath,amscd,amsxtra,latexsym,mathrsfs, epsfig}
\usepackage{dsfont}

\usepackage[english]{babel}
\usepackage[T1]{fontenc} \usepackage[latin1]{inputenc}
\usepackage{graphics,enumerate,amssymb,textcomp,color,graphicx}
\usepackage{color}
\usepackage{cancel}
\usepackage{ulem} 
\usepackage{stmaryrd}

   \def\s{\sigma} 
 \def\phi{\varphi}  \def\b{\beta}
  \def\R{\mathbb{R}} 
 \def\1{\mathbbm{1}}

\theoremstyle{plain} \newtheorem{thm}{Theorem}[section]
\newtheorem{lemma}[thm]{Lemma} 
\newtheorem{hypo}[thm]{Hypothesis} 
\newtheorem{corolary}[thm]{Corollary}
\newtheorem{proposition}[thm]{Proposition}
\newtheorem{corollary}[thm]{Corollary}
\newtheorem{Remark}[thm]{Remark}
\newtheorem{definition}[thm]{Definition}
\def\mathpal#1{\mathop{\mathchoice{\text{\rm #1}}%
    {\text{\rm #1}}{\text{\rm #1}}%
    {\text{\rm #1}}}\nolimits}

\newcommand\ld{\lambda}
\newcommand\sg{\sigma}
\newcommand\te{\theta}

\newcommand\vol{\mathpal{vol}}

\newcommand\dive{\mathpal{div}}

\newcommand\rC{\mathrm{C}}
\newcommand{\TT}{\mathbb{T}}
\newcommand{\st}{\,:\,}
\newcommand\RR{\mathbb{R}}
\newcommand{\pa}{\partial}
\newcommand{\bq}{\begin{eqnarray*}}
\newcommand{\bqn}[1]{\begin{eqnarray}\label{#1}}
\newcommand{\eq}{\end{eqnarray*}}
\newcommand{\eqn}{\end{eqnarray}}
\newcommand{\fo}{\forall\ }

\newcommand{\iy}{\infty}
\newcommand{\ZZ}{\mathbb{Z}}

\definecolor{officegreen}{rgb}{0.0, 0.5, 0.0}

\newcommand{\lt}{\left}
\newcommand{\rt}{\right}

\newcommand{\Ent}{\mathrm{Ent}}
\newcommand{\lin}{\llbracket}
\newcommand{\rin}{\rrbracket}
\newcommand{\lan}{\lt\langle}
\newcommand{\ran}{\rt\rangle}
\newcommand{\ri}{\rightarrow}

\DeclareFontFamily{U}{txsyc}{}
\DeclareFontShape{U}{txsyc}{m}{n}{
   <-> txsyc%
}{}
\DeclareFontShape{U}{txsyc}{bx}{n}{
   <-> txbsyc%
}{}
\DeclareFontShape{U}{txsyc}{l}{n}{<->ssub * txsyc/m/n}{}
\DeclareFontShape{U}{txsyc}{b}{n}{<->ssub * txsyc/bx/n}{}
\DeclareSymbolFont{symbolsC}{U}{txsyc}{m}{n}
\SetSymbolFont{symbolsC}{bold}{U}{txsyc}{bx}{n}
\DeclareFontSubstitution{U}{txsyc}{m}{n}
\DeclareMathSymbol{\df}{\mathrel}{symbolsC}{"42}
\DeclareMathSymbol{\fd}{\mathrel}{symbolsC}{"43}
\DeclareMathSymbol{\lJoin}{\mathrel}{symbolsC}{"58}
\DeclareMathSymbol{\rJoin}{\mathrel}{symbolsC}{"59}

\newcommand{\leqm}{\stackrel{(m)}{\leq}}
\newcommand{\f}[2]{\frac{#1}{#2}}
\newcommand{\rinf}{\rho_{\mathrm{inf}}}

\newcommand{\strict}{}
\newcommand{\strictly}{}
\newcommand{\rai}{r_{\mathrm{int}}}
\newcommand{\rao}{r_{\mathrm{out}}}
\newcommand{\ci}{c_{\mathrm{int}}}
\newcommand{\co}{c_{\mathrm{out}}}
\newcommand{\comment}[1]{}
\newcommand{\cC}{\mathcal{O}}
\newcommand{\wi}{\widetilde}

\newcommand{\aaa}{a} 
\newcommand{\aaaa}{b} 
\newcommand{\CPK}[1]{\textcolor{black}{#1}}

\begin{document}

\allowdisplaybreaks

\title{The stochastic renormalized curvature flow for planar convex sets}
\author[M. Arnaudon]{Marc Arnaudon} \address{Institut de Math\'ematiques de Bordeaux, UMR 5251, \hfill\break\indent 
Univ. Bordeaux, CNRS, Bordeaux INP
} \email{marc.arnaudon@math.u-bordeaux.fr}
\author[K. Coulibaly-Pasquier]{Kol\'eh\`e Coulibaly-Pasquier} \address{Institut \'Elie Cartan de Lorraine, UMR 7502,\hfill\break\indent Universit\'e de Lorraine and CNRS} \email{kolehe.coulibaly@univ-lorraine.fr}
\author[L. Miclo]{Laurent Miclo} \address{{
Toulouse School of Economics, UMR 5314,\hfill\break\indent
 CNRS and Universit\'e de Toulouse}} \email{{laurent.miclo@math.cnrs.fr}}
 
\date{\today\ \emph{ File: }\jobname.tex\\
LM acknowledges fundings from the grants ANR-17-EURE-0010 and AFOSR-22IOE016}
\maketitle

\begin{abstract}
We investigate renormalized curvature flow (RCF) and stochastic renormalized curvature flow (SRCF) for convex sets in the plane.
RCF is the gradient descent flow for logarithm of $\sigma/\lambda^2$ where $\sigma$ is the perimeter and $\lambda$ is the volume. 
SRCF is RCF perturbated by a Brownian noise and has the remarkable property that it can be intertwined with the Brownian motion, yielding a generalization of Pitman "$2M-X$" theorem.
We prove that along RCF,  entropy $\mathcal{E}_t$ for curvature as well as $h_t:=\sigma_t/\lambda_t$ are non-increasing. We deduce  infinite lifetime and convergence to a disk after normalization.
For SRCF the situation is more complicated. The process $(h_t)_t$ is always a supermartingale.  For $(\mathcal{E}_t)_t$ to be a supermartingale, we need that the starting set is invariant by the isometry  group $G_n$  generated by  the reflection with respect to the vertical line and the rotation of angle $2\pi/n$ with $n\ge 3$. But for proving infinite lifetime, we need invariance of the starting set by $G_n$ with $n\ge 7$. We provide the first SRCF with infinite lifetime which cannot be reduced to a finite dimensional flow. Gage inequality plays a major role in our study of the regularity of flows, as well as a careful investigation of morphological skeletons. We characterize symmetric convex sets with star shaped skeletons in terms of properties of their Gauss map. Finally, we establish a new isoperimetric estimate for these sets, of order $1/n^4$ where $n$ is the number of branches of the skeleton.

\noindent MSC2020 primary: 60H15, secondary: 53E10, 35K93, 60J60
\end{abstract}

\section{Introduction} 

The evolution of simple closed surfaces in Euclidean spaces by mean curvature flow has been investigated for a long time, originally motivated by Physics. It is a kind of nonlinear geometrical heat equation.
Here we are interested in the two-dimensional case, known as the curve shortening flow, since it 
can be described as the gradient descent flow for the perimeter.  We will call it the \textbf{curvature flow} (CF), since we have to perturbe it by deterministic and stochastic terms breaking the shortening interpretation.  In 1986, Gage and Hamilton \cite{Gage_Hamilton} proved that starting from any convex smooth simple closed curve, the curvature flow converges in finite time to one point, and the form of the curve becomes circular. In 1987, Grayson \cite{Grayson:87} generalized this result to non necessarily convex starting curve. It is a remarkable fact that no self-intersection occurs during the evolution of the flow.

\subsection{Motivations for renormalized and stochastic curvature flows and main results}

The renormalized curvature flow (RCF) can roughly be defined as the solution to the evolution equation for curves by curvature, to which we add a constant normal field to prevent  implosion.  More precisely we will prove in Lemma \ref{gradfow} that RCF is the gradient descent flow for logarithm of ${\sigma(\partial D)}/{\lambda(D)^2}$ where the considered curve is the boundary $\partial D$ of a bounded domain $D$, $\lambda(D)$ is the volume of the domain and $\sigma(\partial D)$ is the perimeter of the curve. For this flow, self-intersection can occur when the starting curve is not convex. But when the starting curve is convex, we will prove  in Theorem~\ref{TH_section2} that the lifetime of the flow $(\partial D_t)_{t\ge 0}$ is infinite and the curve converges to a circle. Two quantities will be investigated for the convergence: the ratio $ h_t\df{\sigma(\partial D_t)}/{\lambda(D_t)} $ and the entropy $\displaystyle {\rm Ent}_t\df\int_{\partial D_t}\rho_t\log \rho_t $, $\rho_t$ being the curvature at each point of $\partial D_t$. We will prove that these two quantities are non-increasing along the flow (Lemmas~\ref{eq_geo} and~\ref{ent_est}).

One of the main goals of this paper is the investigation of a stochastic renormalized curvature flow (SRCF) in $\RR^2$, where  a one-dimensional normal Brownian noise is added to the evolution of the RCF.  The intensity of the noise is chosen so that the generator of the flow is intertwined with that of the Brownian motion, via a Markov kernel, see \cite{zbMATH07470497} and \cite{arnaudon:hal-03037469}, leading to  nice connections with Bessel-3 processes. When the intertwining is realized through a coupling of the domain-valued process $(D_t)_t$ with  a Brownian motion $(X_t)_{t}$ such that at any time $t\geq 0$, $X_t$ is uniformly distributed
inside $D_t$ conditionally to $(D_s)_{0\le s\le t}$, the construction is a generalization of the famous Pitman "$2M-X$" theorem. An important object in the construction of the coupling $(X_t,D_t)$ is the inner skeleton $S_t$ of $D_t$, which is the singularity set of distance to boundary, inside $D_t$: the evolution equation for $(\partial D_t)_t$ has a component of the drift which is proportional to the local time of $X_t$ at $S_t$, cf.\ \cite{arnaudon:hal-03037469}. A  remarkable fact about the skeleton is that although $(\partial D_t)_t$ has a Brownian noise, $(S_t)_t$ has finite variation. { As we will see in the present paper, the inner skeleton process $(S_t)_t$ also plays a role in the lifetime of $(D_t)_t$. We will prove that starting with a convex subset $D_0$ of $\RR^2$, explosion occurs only when~$\partial D_t$ meets $S_t$ (Theorem~\ref{th_cores}).} We will also prove that similarly to the deterministic situation, the process $h_t$ is a supermartingale (Lemma~\ref{prop-h-sm}). For the entropy being a supermartingale, we will need that $D_0$ is invariant by the linear group $G_n$ generated by the rotation of angle $2\pi/n$ with $n\ge 3$, and the symmetry with respect to an axis (we will choose the vertical one, see Proposition~\ref{Ent}). $G_n$-invariance for any fixed $n\ge 2$  will be proved to be preserved by the flow.  Finally we will prove that $G_n$-invariance of $D_0$ with some $n\ge 7$ implies infinite lifetime for the stochastic renormalized curvature flow (Theorem~\ref{infinite-lifetime}).

In Section~\ref{conv-sym} we investigate some class of convex sets in $\R^2$, which are symmetric with respect to $G_n$ and have star-shaped skeletons. We prove (Proposition~\ref{prop_S}) that they are preserved by all our flows. The last section is devoted to the proof of a new isoperimetric inequality for these classes of convex sets (Proposition~\ref{isoperimetric}). A bound of order $1/n^4$ is obtained. 
\par

\subsection{Parametrization of convex curves and notations}\label{1.2}

We are mainly interested in curves satisfying the following property.\par
\begin{definition} \label{strictconv}
 A simple closed curve is said to be strictly convex when its geodesic curvatures are positive.
\end{definition}
Note that the  inside domain of such a curve is strictly convex in the usual sense i.e.\ it is  strictly contained in one side of any tangent line, except for the contact point, but the converse is not necessarily true, as the curvature may vanish at isolated points.
\par
It is possible to parametrize a simple strictly convex closed curve in $ \mathbb{R}^2$ using the angle $\theta$ between
the tangent vector $T \df (\cos(\theta), \sin(\theta))$ and the oriented $ x $ axis.
The  coordinate $\theta$ will make the equations of our flows simple to analyze, in particular since operators $\partial_{\te} $ and $\partial_t $ will commute, contrary to derivatives with respect to curvilign abscissa $\partial_{s} $ and $\partial_t $, as shown in \eqref{eq_courbure}. 
We will essentially use the one-to-one correspondence in $ \mathbb{R}^2$ between simple strictly convex closed curves (up to translation) and  \strict positive functions $\rho $ that satisfy 
\begin{equation} \label{cores}
 \int_0 ^{2\pi} \frac{\cos(\te)}{\rho(\theta)} \,d\te = \int_0 ^{2\pi} \frac{\sin(\te)}{\rho(\theta)} \,d\te = 0 
 \end{equation}
  as in Lemma 4.1.1 of Gage and Hamilton \cite{Gage_Hamilton}. The function $ \rho$ turns out to be  the curvature of the curve, see also Section \ref{conv-sym}. 

We will derive the evolution equation for the curvature under 
stochastic evolutions of curves, such as stochastic curvature flow (SCF) \eqref{eq_sto_CF}  and SRCF \eqref{eq_sto}.
The positivity of the  curvature as well as Equation~\eqref{cores} are preserved along these equations.
It leads to an alternative definition of the stochastic evolution of a convex curve in terms of the solution of some stochastic partial differential equation, see in particular Theorem~\ref{th_cores}. 
\par\medskip
To fix some notations used throughout the paper, let us recall some notions  associated to a simple  $\rC^2$ closed curve
$C\st \TT\ni u\mapsto C(u)\in\mathbb{R}^2 $, where $\TT\df\RR/(2\pi\ZZ)$.
In this paper, all curves will be closed and immersed. 
\par
The bounded domain whose boundary is $C$ is denoted by $D$. The quantities $\ld (D) $ and $ \sg(C)$ respectively stand for   the volume of $ D$
and the perimeter of $C=\pa D$. We designate by $h(D)$ the isoperimetric ratio $\sigma(\pa D)/\lambda(D)$, not to be confounded with the planar isoperimetric ratio
$\sigma(\pa D)^2/\lambda(D)$.
For any $x\in C$, $\nu_C(x)$ is the outer unit normal vector of the curve $C$ at the point $x$ and $\rho_{C}(x)$ is the corresponding  curvature.\par
When the domain $D(t)$ and its boundary $C_t\df C(t,\cdot)$ depend on time $t\geq 0$, we will  sometimes drop the parameter $D(t)$ or $C_t$ from the notations and even write shortcuts such as  $h(t)$ instead of $h(D(t))$.

\subsection{Alternative approaches}

According to the previous subsection, the shape of a strictly convex curve is given by its curvature function, for instance defined on~$\TT$.
Thus  an evolution of curves, either deterministic or stochastic, can be described by the temporal evolution of its curvature function, which either takes the form of a partial differential equation or  a 
stochastic partial differential equation.
See for instance \eqref{eqcurvtheta} or \eqref{eq_rho} for such evolution equations.
We could then resort to the huge literature on the subject. For instance Lions, Souganidis and their co-authors have a long series of articles on non-linear first or second order stochastic partial differential equations.
But the general equation (1.1) of their latter paper \cite{gassiat2022longtime} does not cover our evolution, due to the fact that they only consider coefficients using the derivatives up to order two of the evolving function, but not the function itself.
\CPK{Furthermore, \cite{gassiat2022longtime} does not consider  non-local terms, such as  $ h(t)$ in Equations \eqref{eq_sto} and \eqref{eq_rho} below.}\par
Another point of view from partial differential equations on curvature type flows consists in interpreting a planar curve $C$ as the level set of a function $u$ defined on~$\RR^2$, say $C=\{(x,y)\in\RR^2\st u(x,y)=0\}$.
When the function $u$ is evolving with time, we get a corresponding evolution of the curve through $C_t\df\{(x,y)\in\RR^2\st u_t(x,y)=0\}$. Curvature flows, usual, renormalized or stochastic, can be represented in this way, with $u$ satisfying a partial differential equation or  a 
stochastic partial differential equation. Indeed, the equations are then homogeneous in space and satisfy the assumptions of Section 2.1 of \cite{gassiat2022longtime}, where $\TT^2$ should be replaced by $\RR^2$. We get a solution $(C_t)_{t}$ defined for all times, but it is not clear if it remains 
non-empty, connected or even
a curve. Furthermore their main asymptotic result  in this setting, Theorem 2.1, does not provide any clue about the stronger asymptotic behaviors we are looking for (spherical shapes), nor about the regularity of the curves or their skeletons.\par
\smallskip
For these reasons, we preferred to use geometric and stochastic methods.

\section{The renormalized curvature flow (RCF)} \label{S2}

Let us  introduce the renormalized evolution  we are interested in.\par
\begin{definition}
 Let  $ C_0\st \mathbb{T}\ni u\mapsto C_0(u) \in\RR^2 $ be a continuous simple and closed planar curve 
 and
 $ C \st [0,T_c) \times \mathbb{T} \to \mathbb{R}^2 $ be a continuous family of simple closed curves indexed by $[0,T_c)$, with $T_c>0$. We say that $C $  starts from $C_0$ and evolves under the renormalized curvature flow (RCF), when it satisfies the following equation
\begin{equation} \left \{  \begin{array}{lcl}\label{eq_det}
 \partial_t C(t,u) &\!\!\!=\!\!\!& [-\rho(C(t,u)) + 2h(D(t)) ]\nu_{C_t}(C(t,u) ),\, \fo (t,u)\in (0,T_c) \times \mathbb{T}\\[2mm]
  C(0,u) &\!\!\!=\!\!\!& C_0(u),\, \fo u\in\TT \\
 \end{array}  \right. \end{equation}
 \end{definition} 
 \par
In the sequel, stronger assumptions than continuity will be made on the initial curve $C_0$ and we will  often refer to:
\begin{hypo} \label{hypoC0} 
 The initial curve $ C_0\st \mathbb{T}\ni u\mapsto C_0(u) \in\RR^2 $ is a simple $C^{2 + \alpha} $ closed and strictly convex planar curve, with $\alpha>0$.
 \end{hypo}
When more regularity is required,  it will be explicitly stated.\par
 \begin{Remark} Since the symbol of Equation \eqref{eq_det} is the same as that of the curvature flow,   short time existence and uniqueness of the solution to  \eqref{eq_det}  hold for simple initial  closed $\mathcal{C}^{\infty} $ curve, see for example \cite{Gage_Hamilton} or \cite{Gage_area}. 
 Existence and uniqueness  still hold, up to the lifetime, if the regularity is relaxed to $\mathcal{C}^{2+\alpha}$ with $\alpha>0$, regularity which is preserved by evolution through \eqref{eq_det}, explaining the above assumption on $C_0$, that will enable us to refer to \textrm{the} solution in the sequel.
 \\
 We need the simplicity of  the curve to avoid ambiguity for   $ h $ (mainly for the interior volume, see Figure \ref{dessin}a at the end of the paper) and to make sure the outer unit normal vector is well-defined.\\
An  alternative proof for existence, using quasi-linear equations, can be found in Chapter 4  of \cite{zbMATH07470497}, Theorem 40 with $B_t =0$ for any $t\geq 0$.  
 \end{Remark}
 \subsection{The main result of this section}
Our main purpose is to investigate the evolution of the curvature function through the RCF. In particular geometrical inequalities concerning planar convex  closed curves will play an important role, as they will provide a priori estimates on the solutions: 
 the Gage inequality \eqref{gage},  involving the non local term $h$, and the usual isoperimetric inequality. The principal result of this section is the following:

 \begin{thm}\label{TH_section2}
Under Hypothesis \ref{hypoC0}, the solution $(C_t)_t $ of equation \eqref{eq_det}  is defined for all $t \in [0,\infty)  $, it remains strictly convex and simple  for all times and is asymptotically circular, the isoperimetric ratio is \strictly decreasing (except for circular starting curves). After  renormalization and translation,  we have the convergence with respect to the Hausdorff metric
   $$ \frac{1}{\sqrt{6t}} [C_t - \ci(t)] \overset{d_{\mathcal{H}}}{\to} \mathcal{C}(0,1) ,$$
to the circle of center $0$ and  radius $ 1$,  where for all $t$, $ \ci(t)$ is the center of an inscribed circle of $C_t $ . 
 \end{thm}
 
 The rest of  Section \ref{S2} is devoted to the proof of Theorem~\ref{TH_section2}.
 
\subsection{Gradient descent flow formulation, and evolution of geometric quantities}  To a solution $(C_t)_{t\in [0,T_c)}$ of \eqref{eq_det}, associate 
 \bq
 \fo t\in [0, T_c),\,\fo u\in\TT,\qquad v(t,u) &\df& \vert \partial_u C(t,u) \vert\eq
  and $s$ the arc-length parametrization, $\partial_s\df  \frac{1}{v} \partial_u $ (equivalently  $ ds = v du $), started at $C(t,s)|_{s=0}=C(t,u)_{u=0}$. 
To prevent the dependence on $t$ of the domain of definition of~$s$, we define $s$ on $\R$  with $\sigma(C_t)$ as period.  Let $T\df  \partial_s C(t,s) $ be the tangent vector of the curve $C(t.)$ at the point $C(t,s) $. Let $\nu(t,s)$ be the unit vector obtained by a rotation of $T(t,s)$ by an angle of $-\pi/2$.  We will always assume that $\nu$ is the outer normal of the curve, up to a change of direction of the parametrization.\par
To 
reinterpret the RCF as a gradient descent flow, let us see the tangent space above a simple closed curve $C$ as the set of $\RR^2$-valued vector fields defined on $C$, and consider the scalar product of two such vector fields $X$ and $X'$ given by
\bq
\lan X, X'\ran_C&\df&\f{1}{\sigma(C)} \int_{C}\lan X,X'\ran_{C(s)}\, ds\eq\par
These definitions provide us with a kind of infinite-dimensional Riemannian structure.
\par
 \begin{lemma}\label{gradfow}
 Equation \eqref{eq_det} is the gradient descent flow of the functional 
  $$ \Psi : D \mapsto { \ln \frac{\s(\partial D)}{\lambda(D)^2}} $$
  relatively to the above structure.
 \end{lemma}
 \begin{proof}
 Let  $ C\st  [0,T) \times \mathbb{T}  \to \mathbb{R}^2 $ be  a family of simple closed curves, such that  $ \partial_t C(t,u) = X(t,u) $ for some smooth $X : [0,T)\times \TT \to \mathbb{R}^2$.  Classical variational computations  show that at any time $t\in[0,T)$,
   $$\frac{d}{dt} \lambda(t) = \int_{C_t} \langle X_t , \nu \rangle  ds $$
and
  $$\frac{d}{dt} \s(t) = \int_{C_t} \langle X_t , \nu \rangle \rho ds, $$ 
  where we recall, in addition to the shortcuts mentioned at the end of Section \ref{1.2},  that $C_t$ is the curve at time $t$ and we denoted similarly $X_t\df X(t,\cdot)$ the vector field on $C_t$, seen as a vector above $C_t$
\par
It follows that for any given $t\in [0,T)$, we have
 \begin{equation*}   \begin{array}{lcl}
\frac{d}{dt} \frac{\s(t)}{\lambda(t)^2} &=& \frac{1}{\lambda(t)^2}  \left(
 \int_{C_t} \langle X_t , \nu \rangle \rho ds - \frac{2\s(t)}{\lambda(t)} \int_{C_t} \langle X_t , \nu \rangle  ds  \right)\\
 &=& \frac{1}{\lambda(t)^2}  
 \int_{C_t} \lt\langle X_t , \left(\rho  - \frac{2\s(t)}{\lambda(t)} \right)  \nu \rt\rangle  ds \\
  &=& \frac{\s(t)}{\lambda(t)^2}  \left(
  \int_{C_t} \lt\langle X_t , \left(\rho  - \frac{2\s(t)}{\lambda(t)} \right)  \nu \rt\rangle  \frac{ds}{\s(t)}  \right).
 \end{array}   \end{equation*} 
 namely
 \bq
 \frac{d}{dt} \Psi(C_t)&=&\lan X_t, \lt(\rho  - \frac{2\s(t)}{\lambda(t)} \rt)\nu\ran_{C_t}.
 \eq
 \par
 Denote $R_t$ the maximum of the r.h.s.\ above all $X_t$ satisfying  $\lan X_t,X_t\ran_{C_t}=1$ and let $\wi X_t$ be a corresponding maximizing vector field.
 The gradient vector field at $C_t$ for the functional $\Psi$ is given by $R_t\wi X_t$.
 \par
 Due to the Cauchy-Schwartz inequality, we get
 \bq
 \wi X_t&=&\f1{\sqrt{\f1{\sigma(t)}\int_{C_t} \lt(\rho_{C_t}(C_t(s))-2\f{\sigma(t)}{\lambda(t)}\rt)^2\, ds}}\lt(\rho  - \frac{2\s(t)}{\lambda(t)} \rt)\nu\\
 R_t&=&\sqrt{\f1{\sigma(t)}\int_{C_t} \lt(\rho_{C_t}(C_t(s))-2\f{\sigma(t)}{\lambda(t)}\rt)^2\, ds}
 \eq
 and it follows that 
 the gradient vector field at $C_t$ for the functional $\Psi$ is $\lt(\rho  - \frac{2\s(t)}{\lambda(t)} \rt)\nu$, i.e.\ the opposite of the vector field appearing in
 \eqref{eq_det}, as required by the gradient descent.
\end{proof} 

Let us start the investigation of the evolution induced by the RCF of some geometric objects:

\begin{proposition}
Under the RCF, we have  

\begin{equation} \label{eq_geo_det} 
\left \{ \begin{array}{lcl}
   \partial_t v &=& -\rho (\rho - 2h) v \\
    \partial_t \partial_s &=& \partial_s \partial_t + \rho( \rho -2h) \partial_s \\ 
        \partial_t  T &=&  -(\partial_s\rho) \nu \\
     \partial_t \nu &= & (\partial_s\rho) T. 
    \end{array} \right. \end{equation}
     
\end{proposition}

\begin{proof}
We differentiate equation \eqref{eq_det} in $u$, and we get:
$$ \begin{aligned}
 \partial_t  \partial_u C &= \partial_u \partial_t C  = -(\partial_u \rho) \nu + (-\rho +2h)\partial_u \nu .\\
\end{aligned} $$
We deduce:
$$ \begin{aligned}
 2 v \partial_t v  &=\partial_t v^2 = \partial_t  \langle \partial_u C ,  \partial_u C  \rangle = 2 \langle \partial_t \partial_u C ,  \partial_u C  \rangle \\
  &= 2 \langle -(\partial_u \rho) \nu + (-\rho +2h)\partial_u \nu, \partial_u C \rangle \\
  &= 2 (-\rho +2h) \langle  \partial_u \nu, \partial_u C   \rangle  =  2 v^2 \rho(-\rho +2h)  .\\
\end{aligned} $$
So we get the first part by identification.
Also by the first computation  
$$ \begin{aligned}
 \partial_t  \partial_s &= \partial_t (\frac{1}{v} \partial_u) = \frac{\rho (\rho - 2h) v}{v^2} \partial_u  +\frac{1}{v} \partial_t \partial_u \\
  &= \rho (\rho - 2h)\partial_s + \partial_s \partial_t\\
\end{aligned},$$
and  
$$ \begin{aligned}
 \partial_t  T =\partial_t  \partial_s C &= \partial_s \partial_t C  + \rho (\rho - 2h)\partial_s C \\
  &= -(\partial_s \rho) \nu + (- \rho + 2h) \partial_s \nu   +  \rho (\rho - 2h)\partial_sC = -(\partial_s \rho) \nu\\
\end{aligned} $$
since $\partial_t \langle \nu , \nu \rangle = 0 $, $ \partial_t  \nu$ is tangential. Also $ \partial_t \langle T , \nu \rangle = 0 $, so we get  the last point from the previous one.  

\end{proof}

We deduce the evolution induced by the RCF of the curvature:

\begin{proposition}
Under the RCF, we have 
\begin{equation} \label{eq_courbure_det}  \partial_t \rho =\partial^2_s \rho + \rho^2 ( \rho -2h)   . 
\end{equation}
\end{proposition}
\begin{proof}
It is a direct consequence of the previous proposition,
$$ \begin{aligned}
 \partial_t  \rho &= \partial_t  \langle T, \partial_s \nu \rangle =   \langle T, \partial_t \partial_s \nu \rangle= \langle T,  \partial_s \partial_t \nu + \rho( \rho -2h) \partial_s \nu \rangle \\ 
 &=\langle T,  \partial_s (\partial_s (\rho) T) + \rho^2( \rho -2h)T  \rangle \\
 &= \partial_s^2  \rho + \rho^2( \rho -2h).
\end{aligned} .$$
\end{proof}

\subsection{A priori estimate of geometric quantities} 
 
 We get the following evolution of geometrics quantities:
\begin{lemma}\label{eq_geo} Assume  the curves of the solution $(C_t)_{t\in[0,T_c)} $ to \eqref{eq_det} remain simple for all $ t \in[0,T_c)$. Then we have for all $ t \in[0,T_c)$,
\begin{enumerate}

\item $ \frac{d}{dt} \sigma(C_t)= - \int \rho^2 ds + \frac{4\pi \sigma(C_t)}{\lambda (D_t)}  ;  $ 
\item $ \frac{d}{dt} \lambda (D_t) = -2\pi + \frac{2 \sigma(C_t)^2}{\lambda (D_t)} ; $
\item $  \frac{d}{dt} h(D(t)) = \frac{d}{dt} \frac{ \sigma(C_t)}{\lambda (D_t)} \le \frac{-12\pi^2}{\sigma(C_t) \lambda (D_t)} \le 0 .$
\end{enumerate}
\end{lemma}
\begin{proof}
Using Gauss-Bonnet Theorem, i.e. for simple  closed curve   $ \int_{0}^{\sigma(C_t)} \rho ds = 2\pi  $, and \eqref{eq_courbure_det} we have: 
\begin{align*}
& \frac{d}{dt} \sigma(C_t)  =  \frac{d}{dt} \int_{0}^{2\pi} v(t,u)du =  \int_{0}^{2\pi} -\rho (\rho - 2h) v du =   \int_{0}^{\sigma(C_t)} -\rho (\rho - 2h) ds\\
& = - \int \rho^2 ds + \frac{4\pi \sigma(C_t)}{\lambda (D_t)}.  \\
\end{align*}
For the second point, we have
\begin{align*}
\frac{d}{dt} \lambda (D_t) = \int_{C_t} \langle \frac{d}{dt} C(t,s) , \nu \rangle ds = \int_{C_t}  - (\rho - 2h) ds = -2\pi + \frac{2 \sigma(C_t)^2}{\lambda (D_t)}. \\
\end{align*}

Let us write  $\sg_t =  \sigma(C_t)  $,  $\ld_t=\lambda (D_t) $ and denote by a dot the derivation with respect to $t$, 
\begin{align*}
\frac{d}{dt} \frac{\sg_t}{\ld_t} &= \frac{1}{\ld_t}(\dot{\sg_t} - \frac{\sg_t \dot{\ld_t}}{\ld_t} ) = \frac{1}{\ld_t}  \left( - \int \rho^2 ds + \frac{4 \pi \sg_t}{\ld_t} - \frac{\sg_t}{\ld_t} (-2\pi + \frac{2\sg_t^2}{\ld_t} )  \right) \\
&\le \frac{1}{\ld_t}   \left( -\frac{4\pi^2}{\sg_t} + \frac{4 \pi \sg_t}{\ld_t} - \frac{\sg_t}{\ld_t} (-2\pi + \frac{2\sg_t^2}{\ld_t} )  \right)\\
&= \frac{-4\pi^2 \ld_t^2 + 6\pi \sg_t^2 \ld_t - 2 \sg_t^4}{\ld_t^3 \sg_t} \\
&= \frac{-2(\sigma_t^2 - 2 \pi\lambda_t )(\sigma_t^2 -\pi \lambda_t )}{\sigma_t \lambda_t^3}\\
&\le \frac{-12\pi^2}{\ld_t \sg_t} \le 0,
\end{align*}
where we have used Cauchy-Schwartz inequality and Gauss-Bonnet Theorem in the second line, and the isoperimetric inequality in the last line.

\end{proof}
\begin{lemma}\label{isope}
Assume  the curves of the solution $(C_t)_{t\in[0,T_c)} $ to \eqref{eq_det} 
 remain convex for all $ t \in [0,T_c) $. Then the isoperimetric ratio is non-increasing, i.e.\ for all $ t \in [0,T_c) $,
$$ \frac{d}{dt} \frac{\sg_t^2}{4\pi \ld_t} \le 0. $$ 
\end{lemma}

\begin{proof}
\begin{align*}
\frac{d}{dt} \frac{\sg_t^2}{4\pi \ld_t} & = \frac{\sg_t}{4\pi \ld_t} ( 2 \dot{\sg_t} - \frac{\sg_t \dot{\ld_t}}{\ld_t} )\\
&= \frac{\sg_t}{4\pi \ld_t}  \left( -2 \int \rho^2 ds +\frac{8\pi \sg_t}{\ld_t } - \frac{\sg_t}{\ld_t} (-2 \pi + \frac{2 \sg_t^2}{\ld_t})  \right)\\
&= \frac{\sg_t}{4\pi \ld_t}  \left( -2 \int \rho^2 ds +\frac{10 \pi \sg_t}{\ld_t } -  \frac{2 \sg_t^3}{\ld_t^2}  \right)\\
&\le \frac{\sg_t}{4\pi \ld_t}  \left( -2 \int \rho^2 ds +\frac{2 \pi \sg_t}{\ld_t }  \right)\\
\end{align*}
where we have use isoperimetric inequality in the last line.
Let us now recall the convex Gage inequality which is proven in \cite{Gage_conv}, and tells us that for convex $ C^2$  plane curves:
\begin{equation}\label{gage}
\frac{\pi \sg_t}{\ld_t} \le \int \rho^2 ds .
\end{equation}

Using this inequality in the above computation we get:
\begin{align*}
\frac{d}{dt} \frac{\sg_t^2}{4\pi \ld_t} \le \frac{\sg_t}{4\pi \ld_t}  \left( - \frac{2 \pi \sg_t}{\ld_t} +\frac{2 \pi \sg_t}{\ld_t }  \right) = 0\\
\end{align*}

\end{proof}
\begin{lemma}\label{def_conv}
Assume  the curves of the solution $(C_t)_{t\in[0,T_c)} $ to \eqref{eq_det} 
remain simple. Then  the deficit of isoperimetry is non-increasing, i.e.:
$$ \frac{d}{dt} (\sg_t^2 - 4\pi \ld_t ) \le 0 .$$
If moreover  the family of curves $C_t $  remain convex for all $ t \in [0,T_c)$ then for all $ t \in [0,T_c) $ we have:
\begin{enumerate}
\item $ \frac{d}{dt} (\sg_t^2 - 4\pi \ld_t ) \le \frac{-2\pi}{\ld_t} ( \sg_t^2 - 4\pi \ld_t), $
\item $ 0 \le (\sg_t^2 - 4\pi \ld_t )  \le (\sg_0^2 - 4\pi \ld_0 )  \left( \frac{ \left(-2\pi + \frac{2\sg_0^2}{\ld_0} \right)t  +\ld_0 }{\ld_0} \right)^{\frac{-2\pi}{-2\pi + \frac{2\sg_0^2}{\ld_0}  }} .$
\end{enumerate}

\end{lemma}

\begin{proof}
By direct computation, and after using Lemma \ref{eq_geo} and similar computation, we have:
\begin{align*}
\frac{d}{dt} (\sg_t^2 - 4\pi \ld_t ) &= 2 \sg_t \dot{\sg_t} - 4\pi \dot{\ld_t} 
= 2 \sg_t \left(- \int \rho^2 ds + \frac{4\pi \sigma_t}{\lambda_t}  \right)-4\pi \left( -2\pi + \frac{2 \sigma_t^2}{\lambda_t} \right) \\
&\le 2 \sg_t \left(- \frac{4\pi^2}{\sg_t} + \frac{4\pi \sigma_t}{\lambda_t}  \right)-4\pi \left( -2\pi + \frac{2 \sigma_t^2}{\lambda_t} \right) \\
&\le 0 .\\
\end{align*}
If moreover  the family of curves $C_t $  remain convex, in the second line of the above computation, we can improve the inequality using \eqref{gage} instead of Gauss-Bonnet Theorem, and we get  for all $ t \in [0,T_c) $:
 \begin{align*}
\frac{d}{dt} \left(\sg_t^2 - 4\pi \ld_t \right)  & \le 2 \sg_t \left(- \frac{\pi \sg_t}{\ld_t} + \frac{4\pi \sigma_t}{\lambda_t}  \right)-4\pi \left( -2\pi + \frac{2 \sigma_t^2}{\lambda_t} \right) \\
&\le \frac{-2\pi}{\ld_t} \left( \sg_t^2 - 4\pi \ld_t\right) .\\
\end{align*}

Using Lemmas \ref{isope} , \ref{eq_geo} and isoperimetric inequality  we deduce that 
$$   6\pi\le -2\pi + \frac{2\sg_t^2}{\ld_t}  \le \dot{\ld_t} \le -2\pi + \frac{2\sg_t^2}{\ld_t} \le  -2\pi + \frac{2\sg_0^2}{\ld_0}, $$
so  for all  $ t \in [0,T_c) $  
$$ 6\pi t + \ld_0 \le \ld_t \le  \left(-2\pi + \frac{2\sg_0^2}{\ld_0} \right)t + \ld_0 .$$  
Hence we get:
$$\frac{d}{dt} (\sg_t^2 - 4\pi \ld_t ) \le \frac{-2\pi}{ \left(-2\pi + \frac{2\sg_0^2}{\ld_0} \right)t + \ld_0} \left( \sg_t^2 - 4\pi \ld_t\right). $$
 
After integration we obtain for all  $ t \in [0,T_c) $:
$$ 0 \le (\sg_t^2 - 4\pi \ld_t )  \le (\sg_0^2 - 4\pi \ld_0 )  \left( \frac{ \left(-2\pi + \frac{2\sg_0^2}{\ld_0} \right)t  +\ld_0 }{\ld_0} \right)^{\frac{-2\pi}{-2\pi + \frac{2\sg_0^2}{\ld_0}  }} .$$

\end{proof}
We deduce the asymptotical shape of the curve $C_t$ as~$t$ goes to infinity.  
\begin{corollary} \label{become_circle}
Assume  the solution   $(C_t)_{t\in[0,+\iy)} $ to \eqref{eq_det} 
is defined for all times and that its curves remain convex. Then we have
$$ \lim_{t \to \infty } \frac{\sg_t^2}{\ld_t} = 4\pi .$$

After  renormalization, the curve $\frac{1}{\sqrt{6t}} [C_t- \ci(t)]$ converges to the circle of center~$0$ and radius~$1$ for the Hausdorff metric, where $\ci(t)$ is the center of an inscribed circle of $C_t $ . 

\end{corollary}

\begin{proof}
Using Lemma \ref{eq_geo} and the isoperimetric inequality, we have $ \dot{\ld_t} \ge 6\pi $, so 
$$ \ld_t \ge 6\pi t + \ld_0 . $$
Using the above Lemma \ref{def_conv} we get 
$$ 0 \le \frac{\sg_t^2 - 4\pi \ld_t }{\ld_t} \le \frac{(\sg_0^2 - 4\pi \ld_0 )  \left( \frac{ \left(-2\pi + \frac{2\sg_0^2}{\ld_0} \right)t  +\ld_0 }{\ld_0} \right)^{\frac{-2\pi}{-2\pi + \frac{2\sg_0^2}{\ld_0}  }}}{6\pi t + \ld_0 } ,$$
and the right hand side goes to $0$ as $t$ goes to infinity.
For the second point, use again  \ref{eq_geo} and the computation above, to deduce that
$$ \dot{\ld}_t  \underset{t \sim \infty }{ \sim} 6\pi,  $$
so $$\ld_t \underset{t \sim \infty }{ \sim} 6\pi t .$$
Since  $  \lim_{t \to \infty } \sg_t^2 - 4\pi \ld_t = 0  $, and using Bonnesen inequality, i.e.
\bqn{Bonne} \frac{\pi^2}{\ld_t} (\rao(t) - \rai(t) )^2 &\le&\left( \frac{\sg_t ^2 }{\ld_t} - 4\pi \right)\eqn
where $ \rao(t) , \rai(t)$ are respectively the outer and the inner radius of the curve $C_t $,
we get  that $$ (\rao(t) - \rai(t) )^2 \le  \frac{\ld_t}{\pi^2} \left( \frac{\sg_t ^2 }{\ld_t} - 4\pi \right) . $$
Let $\ci(t)$ be the center of an inscribed circle of $ C_t$, and $\co(t)$ be a center of a circumscribed circle  of $ C_t$.
Then since $B(\ci(t),\rai(t)) \subset D(t) \subset B(\co(t),\rao(t))  $ we have by Lemma \ref{def_conv}, and isoperimetric inequality that  there exist two positive constants $C>0$ and $\gamma \in\left]0, \frac{1}{6}\right]$ such that ,  
$$ \vert  \ci(t) - \co(t) \vert \le  \rao(t) - \rai(t) \le Ct^{-\gamma} $$
and
$$\rai(t) \le \sqrt{\frac{\ld_t}{\pi}} \le \rao(t) .$$
For two compact sets $ A,B \subset \mathbb{R}^2$ and $ \epsilon \ge 0$  we define $$A_{\epsilon}\df  \{x \in \mathbb{R}^2, d(x,A) \le \epsilon \} $$
and the Hausdorff distance between $ A$ and $B$ by
$$d_{\mathcal{H}}(A,B)\df  \inf \{ r > 0, A \subset B_r \quad and \quad B \subset A_r  \} .$$
Since $B(\co(t), \rao(t)) \subset B(\ci(t), \rai(t))_{2(\rao(t) -\rai(t))}  $ we easily derive that $$ d_{\mathcal{H}}(B(\ci(t), \rai(t)),B(\co(t), \rao(t))) \le 2 ( \rao(t) - \rai(t) ) \le 2Ct^{-\gamma}$$ and $$d_{\mathcal{H}}(D(t),B(\ci(t), \rai(t))) \le 2 ( \rao(t) - \rai(t) ) \le 2Ct^{-\gamma} .$$

So  by Lemma \ref{def_conv}, 
\begin{align*}
&d_{\mathcal{H}}\left(\frac{D(t)- \ci(t)}{\sqrt{6t}},B(0, 1)\right) = \frac{1}{\sqrt{6t}} d_{\mathcal{H}}(D(t),B(\ci(t), \sqrt{6t}))  \\
&\le  \frac{1}{\sqrt{6t}}    \left(d_{\mathcal{H}}(D(t),B(\ci(t), \rai(t))) + d_{\mathcal{H}}( B(\ci(t), \rai(t)),B(\ci(t), \sqrt{6t}))   \right) \\
&\le \frac{1}{\sqrt{6t}}    \left(2 Ct^{-\gamma} + \vert \rai - \sqrt{6t} \vert    \right)  \underset{t \to \infty}{\to}  0 \\
\end{align*}

Similarly  we have that $\frac{1}{\sqrt{6t}} [C_t- \ci(t)]$ converge to the circle of radius $ 1$  and center $0$ for the Hausdorff metric.
\end{proof}
\subsection{Preserving the convexity and lower bound on the curvature}\label{2.4}
We consider here the flow \eqref{eq_det} when the initial curve is strictly convex and simple.
Our purpose is twofold. First 
to show that  strict convexity and simplicity is preserved over its entire lifetime.
Second to prove that the lifetime is infinite, by using intensively some ideas developed in \cite{Gage_Hamilton}.

Let us come back to the angular parametrization recalled in Section \ref{1.2}. 
Usually, the angle $\theta$ depends on $ u$ and $t$. Following \cite{Gage_area} and \cite{Gage_Hamilton}, after adding to the flow \eqref{eq_det}  a tangential perturbation,   the shape of  the curve remains the same, and it is possible    
to find a tangential intensity so that the parameter $\theta$ does not depend on the time.
This  change of coordinate will make the equation simpler to investigate, since the operators $\partial_{\te} $ and $\partial_t $ will commute, contrary to $\partial_{s} $ and $\partial_t $ as shown in \eqref{eq_courbure}. 

Let us quickly recall how to  find the appropriate tangential intensity.
Consider the evolution
\begin{equation}\label{eq_det_pert}
\partial_t C(t,u) = \left(-\rho_t(C(t,u)) + 2h(D(t,.)) \right))\nu + \aaa(u,t) T.
\end{equation}
whose curves $C_t$ have the same shape as those of \eqref{eq_det}, only the parametrizations with respect to $u$ change.
We are looking for a function $\aaa$ making $\theta$ and $t$ independent.
\CPK{Since the mean curvature flow has the regularizing property, see for instance Remark 1.5.3.\ of \cite{MR2815949}, the solutions of Equations \eqref{eq_det} and \eqref{eq_det_pert} are smooth for positive time.}
Differentiating \eqref{eq_det_pert} with respect to $u$ and using $\partial_uT=-v\rho\nu$ together with $\partial_u\nu=v\rho T$  we get 
$$ \partial_t T = \left(- \frac{\partial_u \rho}{v} - \aaa (u,t) \rho \right)\nu $$
and 
$$ \partial_t v = -v\rho^2 + 2h v \rho +  \partial_u \aaa .$$
To make $ \theta$, i.e.\ $T $, independent of $t$ we take 
$\aaa = - \frac{\partial_u \rho}{v\rho} . $\par
Differentiating $\partial_t T  $ with respect to $ u$, we get that $ \partial_t (v\rho) = 0. $
Since $ \partial_{\theta} T = -\nu $,  the chain rule implies $\frac{\partial u}{\partial \theta } \frac{\partial T}{\partial u } = - \nu  $, hence $ \frac{\partial u}{\partial \theta } = \frac{1}{v\rho} $, and we deduce that
$$\frac{\partial }{\partial \theta } = \frac{1}{v\rho} \frac{\partial}{\partial u},  $$
and so $d \theta = \rho v du . $

Since $ \partial_t (v\rho) = 0 $ we have:

\begin{align*}
\partial_t \rho &= - \frac{ (\partial_t v) \rho}{v } = \rho^3 -2h \rho^2 + \frac{\rho}{v} \partial_u (\frac{\partial_u \rho }{v\rho})= \rho^3 -2h \rho^2 + \frac{\rho}{v} \partial_u (\partial_{\theta} \rho )\\
 &= \rho^2 \partial_{\theta} (\partial_{\theta} \rho ) + \rho^2( \rho -2h ).
\end{align*}
{
We record the result in the next lemma.
\begin{lemma}\label{eqrmcf}
When the curvature remains positive, Equation~\eqref{eq_det_pert} for the RCF yields the curvature evolution equation 
\begin{equation}\label{eqcurvtheta}
\partial_t \rho=\rho^2 \partial^2_{\theta} \rho + \rho^2( \rho -2h ).
\end{equation}
\end{lemma}
}
\begin{lemma}\label{conv}
Under Hypothesis \ref{hypoC0},
the solution to  \eqref{eq_det_pert} remains strictly convex and simple up to its lifetime. Moreover, we have
$$ \rho( \te, t) \ge \rho_{\inf} ( 0) e^{- h_0^2 t} $$
where $\rho_{\inf}(0)$ is the minimal curvature of $C_0$.

\end{lemma}

\begin{proof}
Let $ \mathcal{Q}(\te, t) = \rho( \te,t) e^{\mu t} $ for a constant $ \mu  $ that will be chosen later, then $ \mathcal{Q}$ will satisfy the following equation:
\begin{equation}\label{eq_W}
 \partial_t \mathcal{Q} = \rho^2 \frac{\partial^2}{\partial^2 \te} \mathcal{Q} + \mathcal{Q} (\rho^2 - 2h \rho + \mu).
 \end{equation}
The reaction term in the above equation is quadratic in $\rho $, and the discriminant is 
$4 (h^2 - \mu)$.  
Note that the quantity $ h $ in this equation is the same as in \eqref{eq_det}, since the geometric quantities are the same for this equation and \eqref{eq_det_pert}.
Also by Lemma \ref{eq_geo}, assuming for the time being that the curve remains simple,  $h$ is non-increasing, and
$$4 (h^2 - \mu) \le 4(h_0^2 - \mu ) .$$
So choosing $\mu > h_0^2   $ such that this discriminant is negative, the coefficient of $ \mathcal{Q}$ remains \strictly positive. We will apply the maximum principle for this equation. Let
$\mathcal{Q}_{\inf}(t) :=\inf \{ \mathcal{Q}(\te,t ), \, 0\le \te \le 2\pi\} $. The proof is by contradiction,  suppose that there exist $ 0 < \eta <\mathcal{Q}_{\inf}(0) $ and $ t > 0$ such that $ \mathcal{Q}_{\inf}(t)= \eta$, let $ t_0$ be the first time such that $\mathcal{Q}_{\inf}(t_0)= \eta .$  This minimum is achieved at some point $ \te_0$, and at this point:\\
 $\partial_t \mathcal{Q}(\te_0,t_0) \le 0$, $\frac{\partial^2}{\partial^2 \te} \mathcal{Q} (\te_0,t_0) \ge 0 $, and $\mathcal{Q}(\te_0,t_0)= \eta.$
 This is in contradiction with  Equation  \eqref{eq_W}.
 Hence for all $ 0<t $, $\mathcal{Q}_{\inf}(t) \ge \mathcal{Q}_{\inf}(0)$ and 
 $$ \rho_{\inf}(  t) \ge \rho_{\inf} ( 0) e^{-\mu t},$$ where $\rho_{\inf} (t) = \inf \{ \rho(\te,t ), \, 0\le \te \le 2\pi\}  $
 so $\rho( \te, t) \ge \rho_{\inf} ( 0) e^{-\mu t} $ for all $\mu > h_0^2   $. Hence if $C_0 $ is strictly convex then $C_t $ remains strictly convex at any time~$t$ up to which it is defined.
  
  Above, while resorting to   Lemma \ref{eq_geo},  we assumed  that the curves solution to \eqref{eq_det_pert} remain simple. 
  Let us show it is true, through the same
 argument  by contradiction. Let $T_s $ be the first time   the curve stops to be simple.
 If $T_s$
  occurs strictly before the maximal lifetime of  \eqref{eq_det_pert}, i.e. $ 0< T_{s} \le t* < T_c$,  the same computation as above shows the curvature $\rho $ is \strictly positive until $T_{s}$, and $\rho$ and all its derivatives are bounded in $[0,t*] $. So there exist a limiting curve as $t$ goes to $ T_s$, that is smooth and has \strictly positive curvature, due to Arzel\`a-Ascoli theorem. Taking into account  Lemma~4.1.1 and Theorem 4.1.4 in \cite{Gage_Hamilton}, we know that  positive curvature and  
  \eqref{cores}
 characterizes  simple close strictly convex curves. So the limiting curve  is simple and strictly convex, contradicting the non-simplicity at time $T_s$.
   \end{proof}
 \begin{corollary}\label{conv_plus}
Still under Hypothesis \ref{hypoC0}, the bound of Lemma \ref{conv} can be improved into
$$ \rho( \te, t) \ge \frac{1}{\frac{1}{\rho_{\inf}(0)} + \frac{\sg_0^2}{6\pi^2 \ld_0} \sqrt{24 \pi^2 t + 4\pi\ld_0 }} .$$
\end{corollary}  
 \begin{proof}
 Using Lemma \ref{conv}, we have that $\rho > 0 $ so we can define $W\df e^{- \frac{1}{\rho} + \int_0^t 2h_s \, ds} $.
We compute
 \begin{align*}
\partial_t W &= (\partial_t e^{- \frac{1}{\rho}} ) e^{ \int_0^t 2h_s \, ds} + 2h W\\
  &=  ( \frac{\partial_t \rho}{\rho^2} + 2h) W\\
  &=  (\partial^2_{\te} \rho + \rho )W.
 \end{align*}
 We will apply the maximum principle for this equation. Define
\bq W_{\inf}(t) &\df&\inf \{ W(\te,t ), \, 0\le \te \le 2\pi\}. \eq The proof is by contradiction,  suppose that there exists $ 0 < \eta <W_{\inf}(0) $ and $ t > 0$ such that $ W_{\inf}(t)= \eta$. Let $ t_0 >0$ be the first time such that $W_{\inf}(t_0)= \eta .$  This minimum is achieved at some point $ \te_0$, and at this point (since $W$ is a non-decreasing function in $ \rho$):\\
 $\partial_t W(\te_0,t_0) \le 0$, $\frac{\partial^2}{\partial^2 \te} W (\te_0,t_0) \ge 0 $, and $W(\te_0,t_0)= \eta$, so $\frac{\partial^2}{\partial^2 \te} \rho (\te_0,t_0) \ge 0 $.
 This is in contradiction with the equation satisfied by $W$.
 Hence for all $ 0<t $, $W_{\inf}(t) \ge W_{\inf}(0)$, so 
 $$  e^{- \frac{1}{\rho} + \int_0^t 2h_s \, ds}  \ge e^{- \frac{1}{\rho_{\inf}(0)}}. $$ 
 By Lemma \ref{isope}, we have $$h(t)\df \frac{\sg_t}{\ld_t} \le \frac{\frac{\sg_0^2}{ \ld_0}}{\sg_t}. $$  
 By isoperimetric inequality we have $\sqrt{4\pi \ld_t} \le \sg_t $ hence by Lemma  \ref{eq_geo} we have $\sqrt{4\pi (6\pi t + \ld_0 )} \le \sg_t $ and 
$$h(t)  \le \frac{\frac{\sg_0^2}{ \ld_0}}{\sqrt{4\pi (6\pi t + \ld_0 )}} .$$ 
This yield 
$$ \int_0^t 2h(s) \, ds \le \frac{\sg_0^2}{6\pi^2 \ld_0} [\sqrt{24 \pi^2 t + 4 \pi\ld_0 } -\sqrt{ 4\pi\ld_0 } ] ,$$
and $$ \rho(t) \ge \frac{1}{\frac{1}{\rho_{\inf}(0)} + \frac{\sg_0^2}{6\pi^2 \ld_0} \sqrt{24 \pi^2 t + 4\pi\ld_0 }}.$$ 
 \end{proof} 
\subsection{Upper bound on the curvature and long time existence of the flow for strictly  convex simple initial curves}
We will now show that under Hypothesis~\ref{hypoC0}, the solution of \eqref{eq_det} can be defined for all times, by first establishing a uniform bound (depending on time horizon) of the maximum of curvature.
Similarly to \cite{Gage_Hamilton} we define the pseudo-median of the curvature:
$$ \rho^{*}(t) = \sup \{ \beta, \rho(\te,t )> \beta \text{ on some interval of length } \pi \}. $$
\begin{lemma}\cite{Gage_Hamilton}\label{geo}
If a planar convex closed curve encloses an area $ \ld$ and has length $\sg $ then $ \rho^{*}  \le \frac{\sg}{\ld}. $
\end{lemma}
Following the above lemma we have:
\begin{corolary}
Under Hypothesis~\ref{hypoC0}, consider a solution to Equation \eqref{eq_det} defined on $[0, T ) $. Then for all $  t \in[0, T) $, we have
 $$ \rho^{*} (t) \le h \le h_0. $$
\end{corolary}
\begin{proof}
We use Lemma \ref{conv} to get the convexity until $ T$, { Lemma~\ref{eq_geo} for $h$ non-increasing,  } and  Lemma \ref{geo} to conclude.
\end{proof}

\begin{lemma}[Entropic estimate]\label{ent_est}
Under Hypothesis~\ref{hypoC0}, consider a solution to Equation \eqref{eq_det} defined on $[0, T ) $. Then 
 \bqn{Ent2} \Ent(t)&\df&\int_0^{2\pi} \log(\rho (\te,t) ) \,  d\te \eqn
 is non-increasing on $[0,T)$.\\
 Moreover we have the following estimates for $t\in[0,T)$:

\bq
2\pi \left[ \log{\rho_{\inf}(0) - h_0^2 t}\right] &\le & \Ent(t)\ \le\  \Ent(0)+  \frac{\pi\rho_{\inf}^2(0)}{2h_0^2} \left[ e^{-2h_0^2 t} -1\right]  \eq
and there exists five explicit constants $c_0, c_1, \tilde{c}_0,\tilde{c}_1,\tilde{c}_2 $,  only depending on the geometry of the initial curve, such that

\bq-2\pi \log\left( \tilde{c}_0 + \sqrt{ \tilde{c}_1t +\tilde{c}_2 }\right)&\le & \Ent(t)\ \le\   \Ent(0)  -\frac{2\pi}{c_1} \log \left(\frac{c_1t+c_0}{c_0} \right) . \eq

\end{lemma}
\begin{proof}
The proof is an adaptation of the proof in \cite{Gage_Hamilton}. We will just point the differences.
After an integration by part we have:
$$ \frac{d}{dt} \int_0^{2\pi} \log(\rho (\te,t) ) \,  d\te = \int_0^{2\pi} -\left(\frac{\partial}{\partial \te} \rho \right)^2  + \rho( \rho -2h )   \, d\te .$$
Let write the open set $ U =\{\te, \rho(\te,t) > \rho^*(t) \}$ as an union of disjoint interval $ I_i$, by definition of the median the length of all $ I_i$ is smaller than $\pi $, and 
\begin{align*}
\int_{I_i}  -\left(\frac{\partial}{\partial \te} \rho \right)^2  + \rho( \rho -2h )   \, d\te & = 
\int_{I_i}  -\left(\frac{\partial}{\partial \te} (\rho - \rho^*) \right)^2  + \rho^2 \, d\te  -2h  \int_{I_i}  \rho  \, d\te \\
& \le \int_{I_i}  - (\rho - \rho^*) ^2  + \rho^2 \, d\te  -2h  \int_{I_i}  \rho  \, d\te \\
&= \int_{I_i}  2\rho\rho^* - (\rho^*)^2\, d\te  -2h  \int_{I_i}  \rho  \, d\te \\
&\le (2\rho^* -2h) \int_{I_i}  \rho d\te - (\rho^*)^2 \int_{I_i}   \, d\te, \\
\end{align*}
where in the second line we used Wirtinger  \cite{Gage_Hamilton} inequality (recall that on the boundary of $ I_i$ , $ \rho = \rho^*$).

On the complement of $ U $ we have:
\begin{align*}
\int_{U^c}  -\left(\frac{\partial}{\partial \te} \rho \right)^2  + \rho( \rho -2h )   \, d\te & \le (\rho^* -2h) \int_{U^c}  \rho d\te \\
 &\le   (2\rho^* -2h) \int_{U^c}  \rho d\te - \rho^*  \int_{U^c}  \rho d\te .\\
\end{align*}

Hence using Lemma \ref{geo}, and  $ \rho^*(t) \ge \rho_{\inf}(t) $ we get
\begin{align*}
 \frac{d}{dt} \int_0^{2\pi} \log(\rho (\te,t) ) \,  d\te  &\le  (2\rho^* -2h) \int_{0} ^{2\pi} \rho d\te - \rho^*  \int_{U^c}  \rho d\te - (\rho^*)^2 \int_{U}   \, d\te   \\ 
 &\le -2\pi \rho_{\inf}^2(t). \\
   \end{align*}
So the first part of the lemma is proved. Lemma \ref{conv} yields after integration:
$$\int_0^{2\pi} \log(\rho (\te,t) ) \,  d\te  \le \int_0^{2\pi} \log(\rho (\te,0) ) \,  d\te     + \frac{\pi\rho_{\inf}^2(0)}{h_0^2} \left[ e^{-2h_0^2 t} -1\right],$$ and Corollary \ref{conv_plus} yields the second estimate.

For the lower bound use the bounds of Lemma \ref{conv} and \ref{conv_plus}.
\end{proof}
\begin{proposition}\label{inv_Poin}
Under Hypothesis~\ref{hypoC0}, consider a solution to Equation \eqref{eq_det} defined on $[0, T ) $. Then 
 there exists a constant $c_0$ that depends only on the initial curve such that:
$$
\int_0^{2\pi} \left(\frac{\partial } {\partial \te} \rho \right)^2 \,d\te \le \int_0^{2\pi} \rho^2\,d\te -4h_t \int_0^{2\pi} \rho \,d\te + c_0 .
$$
\end{proposition}

 \begin{proof}
 By Lemma \ref{conv} we know that the curvature $\rho $ remains \strictly positive during the existence of the flow, so we can compute: 
\begin{align*}
&\frac{d}{dt} \int_0^{2\pi}   \left(\rho^2 - (\frac{\partial } {\partial \te} \rho )^2 -4h \rho  \right) \, d\te \\
&=2  \int_0^{2\pi} \frac{d  \rho}{dt}   \left( \rho +( \partial^2_{\theta} \rho) -2 h   \right) \, d\te - 4\frac{d h}{dt}  \int_0^{2\pi} \rho  \, d \te \\ 
&= 2 \int_0^{2\pi} \left( \frac{\partial_t \rho}{ \rho} \right)^2   \, d\te - 4\frac{d h}{dt}  \int_0^{2\pi} \rho  \, d \te \\ 
&\ge - 4\frac{d h}{dt}  \int_0^{2\pi} \rho  \, d \te \\ 
&\ge  \frac{48 \pi^2}{\ld \sg} \int_0^{2\pi}  \rho  d\te  > 0 ,\\
\end{align*}

where we have used an integration by part on the second line, the equation of the curvature at the third line, and Lemma \ref{eq_geo} at the last line.

Integrating the last inequality we get for all $t \in [0,T)  $,
\begin{align*}
 \int_0^{2\pi}   \left(\rho^2 - (\frac{\partial } {\partial \te} \rho )^2 -4h \rho  \right)_{ \vert_{t}} \, d\te \ge   \left( \int_0^{2\pi}   \left(\rho^2  - (\frac{\partial } {\partial \te} \rho )^2 -4h \rho  \right) \, d\te  \right)_{ \vert_{0}} = -c_0\\
\end{align*}
We obtained the last inequality using \eqref{gage} and the upper bound of the volume during the flow (Lemma \ref{eq_geo}).
\end{proof}

\begin{proposition}\label{bound-rho}
Under Hypothesis~\ref{hypoC0}, consider a solution to Equation \eqref{eq_det} defined on $[0, T ) $. If $T < \infty$,
$$  M = \sup \{ \rho(\te, t),(\te, t) \in\TT  \times [0,T)  \} $$
 is bounded .
\end{proposition}
\begin{proof}
For $ t < T$ let $$M_t = \sup \left\{ \rho(\te, s),(\te, s) \in\TT  \times [0,t]  \right\}  .$$
There exists $ (\te_1 , t_1)  \in\TT \times [0,t]   $ such that $M_t = \rho(\te_1, t_1) $. 
Then for all $\te_2 \in\TT  $ we have:
\begin{align*}
\vert \rho(\te_1, t_1) - \rho(\te_2, t_1)\vert &= \left\vert\int_{\te_2}^{\te_1} \frac{\partial } {\partial \te} \rho    \, d\te \right\vert \\
&\le  \sqrt{\int_0^{2\pi} (\frac{\partial } {\partial \te} \rho )^2 \,d\te } \sqrt{\vert \te_1 - \te_2 \vert} \\
&\le \sqrt{2 \pi M_t^2 + c_0 }  \sqrt{\vert \te_1 - \te_2 \vert} \\
&\le M_t \sqrt{2 \pi + \frac{\vert c_0 \vert}{ M_t^2} }  \sqrt{\vert \te_1 - \te_2 \vert} ,\\
\end{align*}
where we have used Proposition \ref{inv_Poin} at the third line, and $\vert \te_1 - \te_2 \vert$ is the usual distance in $\TT  $.
We also have,
$$ M_t \ge  \rho_{sup} ( 0)  . $$
So $$ \rho(\te_1, t_1) - \rho(\te_2, t_1)   \le c_1 M_t \sqrt{\vert \te_1 - \te_2 \vert} , \\ $$

where $ c_1\df  \sqrt{2 \pi + \frac{\vert c_0 \vert}{  \rho^2_{sup} ( 0) }} .$

It follows that for all  $\te_2 $:
$$  \rho(\te_2, t_1)  \ge M_t -c_1 M_t \sqrt{\vert \te_1 - \te_2 \vert} $$ 
and 
\begin{align*}
& \int_0^{2\pi} \log(\rho (\te,t_1) ) \,  d\te \\
 &= \int_{ \vert \te_1- \te \vert \le (\frac{1}{2c_1} )^2} \log(\rho (\te,t_1) ) \,  d\te  +  \int_{ \vert \te_1- \te \vert \ge (\frac{1}{2c_1} )^2} \log(\rho (\te,t_1) ) \,  d\te  \\
& \ge \log \left(\frac{M_t}{2} \right)\frac{1}{2c_1^2}  + \left(2 \pi - \frac{1}{2c_1^2} \right) \log  \left( \rho_{\inf} (0) e^{- h_0^2 T}   \right)\\
\end{align*}
Use Lemma \ref{ent_est} we obtain for all  $ t \in [0,T)$
\begin{align*}
\log (M_t) \le c_2(T)
\end{align*} 
where $$c_2(T)=2c_1^2  \left(  \int_0^{2\pi} \log(\rho (\te,0) ) \,  d\te \ - \left(2 \pi - \frac{1}{2c_1^2} \right) (\log  \left( \rho_{\inf} (0)) -  h_0^2 T   \right)  \right) $$
is a function that only depends on the final time and the initial curve.
So $  M = \sup \left\{ \rho(\te, t),(\te, t) \in\TT \times [0,T)  \right\} $ is bounded for $ T < \infty$.
\end{proof}

\begin{proposition}\label{bound_derivative}
Under Hypothesis~\ref{hypoC0}, consider a solution to Equation \eqref{eq_det} defined on $[0, T ) $. If $T < \infty$, 
 for all $ n \in \mathbb{N}$, we have
$$  M^{(n)} = \sup \left\{ \frac{\partial^n \rho}{\partial^n \te} (\te, t),(\te, t) \in\TT \times [0,T)  \right\} $$
 is bounded.
\end{proposition}
 \begin{proof}
 Following \cite{Gage_Hamilton} section 4.4, we will first prove that $\frac{\partial\rho}{\partial \te}  $ is bounded along the flow.
 By direct computation,  $\frac{\partial \rho}{\partial \te}  $ satisfies:
 \begin{align*}
\frac{\partial}{\partial t}    \frac{\partial \rho}{\partial \te} = \rho^2 \partial^2_{\te} \left(  \frac{\partial \rho  }{\partial \te} \right) + 2\rho \left(\frac{\partial\rho }{\partial \te}  \partial_{\te} \left(\frac{\partial \rho}{\partial \te} \right) \right) +\left[ 3\rho^2 -4h\rho \right]\frac{\partial \rho}{\partial \te} \\
 \end{align*}
 By Lemma \ref{eq_geo} and Proposition \ref{bound-rho}, $[ 3\rho^2 -4h\rho ] $ is bounded so by the maximum principle $\frac{\partial \rho}{\partial \te }$ is bounded for all $ t \in [0,T) $. (this is easier than in the proof of \ref{bound-rho}).
 
 Since the equation for  $\frac{\partial^2 \rho}{\partial^2 \te }$  contains a quadratic term it seems not clear to directly use the maximum principle. To control it we will  proceed as follows.

 With the same computation as in \cite{Gage_Hamilton} Lemma 4.4.2 we see that modulo a additional term that comes from the non local term  $ h $, which is bounded, we show that $\displaystyle \int \left(\frac{\partial^2 \rho}{\partial^2 \te }\right)^4\, d\te$ is bounded on finite intervals of time (i.e. when $T < \infty$).
 Let us prove this property. To present the difference with the computation  in \cite{Gage_Hamilton}, we integrate by part. We get:
 \begin{align*}
 & \frac{\partial}{\partial t}  \int_0^{2\pi} \left(\frac{\partial^2 \rho}{\partial^2 \te }\right)^4\, d\te = 4 \int_0^{2\pi} \partial^2_{\te}  \left(\rho^2 \partial^2_{\te} \rho + \rho^2 (\rho - 2h)  \right)\left( \partial^2_{\te} \rho \right)^3 \, d\te \\
& = -12 \int_0^{2\pi}    \left( \rho^2\partial^3_{\te} \rho + 2\rho \partial_{\te} \rho      \partial^2_{\te} \rho+ (3\rho^2 - 4h\rho)\partial_{\te} \rho  \right) ( \partial^3_{\te} \rho ) ( \partial^2_{\te} \rho )^2 \, d\te. \\
   \end{align*} 
 To simplify notations let us use $\rho '\df \partial_{\te}\rho $, in the above computation:
 \begin{align*}
  \frac{\partial}{\partial t}  \int_0^{2\pi} ( \rho''  )^4\, d\te  &= -12 \int_0^{2\pi}  
  \Big( \rho^2 (\rho'')^2(\rho''')^2 + 2\rho \rho' (\rho'' )^3  \rho'''\\
 & + 3\rho^2 \rho' (\rho'')^2 \rho''' - 4h \rho \rho' (\rho'')^2 \rho'''  \Big) \, d\te. 
  \end{align*} 
 We use the inequality $ab \le \frac{1}{4\epsilon}a^2 + \epsilon b^2 $ to bound the three last terms by the first one and some additional terms, after choosing $\epsilon$  to control the sign of the first term. We obtain that there exist $c_1,c_2,c_3 $ such that:
  \begin{align*}
  \frac{\partial}{\partial t}  \int_0^{2\pi} \left( \rho''  \right)^4\, d\te  &\le c_1 \int_0^{2\pi}(\rho')^2 (\rho'')^4 \, d\te  +  c_2 \int_0^{2\pi} (\rho)^2  (\rho')^2 (\rho'')^4 \, d\te \\
  &  + c_3 \int_0^{2\pi}   (\rho')^2 (\rho'')^2 \, d\te  \\
  \end{align*} 
  where the constant $ c_3$ depends on $h(0)$ which is the maximum of $h $ by Lemma~\ref{eq_geo}. Since $\rho $ is bounded by proposition \ref{bound-rho} and $\rho' $ is bounded we deduce   from the above inequality and   Cauchy-Schwarz inequality that
  $  \int_0^{2\pi} \left( \rho''  \right)^4\, d\te$ remains bounded on $[0,T) $.

 By the same kind of computation, we will show that   $ \int \left(\frac{\partial^3 \rho}{\partial^3 \te }\right)^2\, d\te$ remains bounded for  $ t \in [0,T)$ (when $T < \infty$). After a integration by part, we have:
 \begin{align*}
& \frac{d}{dt} \int (\rho''')^2\, d\te = 2  \int (\rho''') \left( \rho^2 \rho'' + \rho^2 ( \rho -2h)  \right)'''    \, d\te  \\
 &= -2  \int (\rho'''') \left( \rho^2 \rho'' + \rho^2 ( \rho -2h)  \right)''    \, d\te  \\
 &  = -2 \int (\rho'''')  \Big(\rho^2 \rho'''' +2 \rho\rho' \rho'''+2\rho(\rho'')^2 + 2(\rho')^2\rho'' + 3 \rho^2 \rho'' \\
 &\quad \quad \quad \quad \quad+ 6 \rho(\rho')^2 - 4h (\rho')^2 -4h \rho(\rho'')  \Big)    \, d\te . \\ 
 \end{align*}  
Using again the inequality $ab \le \frac{1}{4\epsilon}a^2 + \epsilon \b^2  $ for a well choosed $\epsilon $  to bound the seven last terms by the first one and some additional terms, we get that there exist $c_1,c_2,c_3,..c_6,c_7 $ (that all depend on $\epsilon$, and $c_6,c_7$ depend also on $h_0 $, the upper bound of $h$ by Lemma \ref{eq_geo}) such that:
 \begin{align*}
& \frac{d}{dt} \int (\rho''')^2\, d\te\\ &\le c_1 \int (\rho \rho''')^2\, d\te +  c_2 \int (\rho'')^4\, d\te +  c_3 \int \left(\frac{(\rho')^2 \rho''}{\rho}\right)^2\, d\te \\
&   +  c_4 \int (\rho \rho'')^2\, d\te +  c_5 \int (\rho')^4\, d\te+ c_6 \int \left(\frac{(\rho')^2}{\rho}\right)^2 + c_7 \int (\rho'')^2 \, d\te. \\
  \end{align*} 
Since on finite intervals of time, $\rho $ is bounded by Lemma \ref{bound-rho} and by the computation above $\rho' $ and $\int (\rho'')^4 \,d \te $ are bounded, and the lower bound $\rho \ge \rho_{\inf}(0) e^{-h_0^2t} $ (Lemma \ref{conv}), using Cauchy-Swartz inequality we get for other constants (that depend on the time horizon $T$):

$$ \frac{d}{dt} \int (\rho''')^2\, d\te \le c_0 \int (\rho''')^2\, d\te + c_2 .$$

 Hence $ \int \left(\frac{\partial^3 \rho}{\partial^3 \te }\right)^2\, d\te$ remains bounded for  $ t \in [0,T)$, and so $ \frac{\partial^2 \rho}{\partial^2 \te }$ is bounded, using fundamental calculus theorem, or Sobolev inequality in $\mathbb{S}^1 $.

  For all $ n\ge 3 $ the equation for $\frac{\partial^n \rho}{\partial^n \te  }=\fd \rho^{(n)}$ writes 
  \begin{align*}
  \frac{d}{dt} \rho^{(n)} &= \rho^2 \rho^{(n+2)} + 2n\rho\rho'\rho^{(n+1)} + \bigr[2\rho\rho'' +3\rho^2-4h\rho \\
  &+ (n)(n-1) (\rho'+ \rho\rho'' )\bigl]\rho^{(n)} + P\left(h,\rho,\rho',...,\rho^{(n-1)}\right) \\
  \end{align*}
  where $P\left(h,\rho,\rho',...,\rho^{(n-1)}\right) $ is a  polynomial in $\left(h,\rho,\rho',...,\rho^{(n-1)}\right). $
Since $ \rho, \rho', \rho''$ are bounded by the computation above and $ h$ is bounded by Lemma \ref{eq_geo}, we can apply the maximum principle to get an exponential bound for $\rho''' $, so $\rho'''  $ is bounded (when $T$ is finite). Using the above equation for $\rho^{(n)}$, we get by induction and maximum principle that { for all $n$,  $\rho^{(n)} $ is uniformly bounded  on $  [0,T)$} when $ T < \infty$.

 \end{proof}
 
\subsection{Proof of  Theorem \ref{TH_section2}} 
 \leavevmode\par
 \medskip
 We  prove the long time existence of the flow. 
   Assume that  the starting curve $C_0 $ is simple and strictly convex, and the flow exists for all $t \in [0,T) $. 
 Then by Lemma \ref{conv} we know that the solution of \eqref{eq_det_pert} remains convex and simple during the flow. 
  Since the solution of  \eqref{eq_det_pert} has similar shape as the solution of \eqref{eq_det} (just the parametrisation changes) we know that the solution of \eqref{eq_det} remains convex  and simple,  so the  quantity $ h $ remains defined for all $t \in [0,T)$. Using Lemma \ref{eq_geo} we get that the quantity $h $ remains bounded as soon as the flow exists.
  By Propositions \ref{bound-rho}  and \ref{bound_derivative} we know that $\rho$ and all spacial derivatives of $\rho$ are bounded in $[0,T) $, if $T < \infty $, hence { the same} for time derivative of $ \rho$. So by Arzel\`a-Ascoli Theorem, $ \rho $ converges to a $C^{\infty} $ function $\rho(T,.) $ as $ t \to T $. Using equation \eqref{eq_det} there exists a limiting curve $C_T $ and this limiting curve is associated to $ \rho (T, .)$ (in the sense of Lemma 4.1.1 in \cite{Gage_Hamilton}), so  $C_T $ is  strictly convex and simple. By the small time existence if $T < \infty $, we can extend the time interval on which the solution is defined using the solution that starts at $C_T$. This proves that the solution of \eqref{eq_det} starting with at a simple strictly convex  curve exists for all time. 
  By Lemma  \ref{isope},  the isoperimetric ratio is non-increasing. It is in fact \strictly decreasing until the curve becomes a circle (take strict inequality in the isoperimetric inequality  in the proof of Lemma \ref{isope}), but if it would becomes a circle in finite time we could reverse the flow and get that the starting curve is a circle. So the isoperimetric ratio is \strictly decreasing   unless if the starting curve is circular.  
  Using Lemma \ref{def_conv} we get that the deficit of isoperimetry converges to $ 0$ polynomially, and Corollary \ref{become_circle} shows that the family of curves becomes more and more circular, and the isoperimetric ratio  decreases to $ 4\pi$. Also this corollary yields, with Bonnesen inequality, the convergence after normalization to a circle, i.e.
  $ \frac{1}{\sqrt{6t}} [C_t - \ci(t)] $ converges to circle of center $0$ and  radius $ 1$ for Hausdorff metric, where $ \ci(t)$ is an center of a inscribed circle of $C_t$. 
\par\hfill $\square$\par\medskip
\begin{Remark}
When the starting curve $C_0$ is convex and not simple, recall Figure~\ref{dessin}a at the end of the paper, the flow is not well defined. And when the starting curve is simple but non-convex, the existence in long time can be problematic, see Figure~\ref{dessin}b.
\end{Remark}

\section{The stochastic renormalized curvature flow (SRCF)~in~$\mathbb{\R}^2$}

\subsection{Equations of geometric quantities along the stochastic flow of curves in~$\mathbb{\R}^2$} 

Let us introduce a noisy extension of the RCF. We need a standard  Brownian motion $(B_t)_{t\geq 0}$ defined on a filtered probability space. All subsequent stopping times are with respect to the underlying filtration. Except when they are parametrized by the arc-length $s$, all the closed curves are assumed to be continuous and parametrized by a parameter belonging to $\TT$. 
\par
\begin{definition}
 Let 
 $ C \st [0,\tau) \ni t\mapsto C_t $ be a continuous family of simple closed curves indexed by $[0,\tau)$, where $\tau$ is a positive stopping time. We say that $C$ evolves under the renormalized stochastic curvature flow (SRCF), if it satisfies the following equation for any $ t\in[0,\tau)$,
 \bqn{eq_sto}
 d_t C(t,u) & = & \lt([-\rho_t(C(t,u)) + 2h_t ]dt + \sqrt{2}dB_t \rt) \nu_t(C(t,u))   \eqn
 (where  $h_t$, $\nu_t$ and $\rho_t$ are our usual shortcuts, cf.\ Section \ref{1.2}).
\end{definition} 

 When convenient without any possible confusion, the index $t\geq 0$ will  be omitted. 
An important goal of this paper is to show that the above equation  admits a solution in the whole temporal interval $\RR_+$ under some assumptions.
 Until it is proven, when we consider a time $t\geq 0$, it will be implicitly assumed that $t$ is smaller than the stopping time $\tau$, that will be called a \textbf{lifetime} for \eqref{eq_sto}. The curve $C_0$ will be referred to as the initial curve.
  \begin{Remark} For the short time existence and uniqueness up to $\tau$ of the solution to \eqref{eq_sto}, 
  we refer to Theorem~40  in \cite{zbMATH07470497}, where the authors used Doss-Sussman method, the theory of  quasi-linear equations, as well as the inverse function theorem.
  \par Concerning the regularity,  for any $0<\alpha < 1$,   the solution of \eqref{eq_sto} is  $\rC^{\alpha / 2, \infty}$ if the starting curve $C_0 $ is smooth. In fact it is enough  to consider that  the starting curve $C_0 $ are $\rC^{\alpha  + n}$ , for $ n \ge 2$ and $0<\alpha < 1$, to get the $\rC^{\alpha / 2,\alpha+ n}$ regularity of the solution of \eqref{eq_sto}  (cf.\ Chapter 8 of Lunardi \cite{MR3012216} and \ Chapter 4 in \cite{zbMATH07470497}). So, to justify the existence of all the derivatives one may need, it is sufficient to take $ C_0$  regular enough, but we will not insist on the regularity of $C_0$ in the rest of the paper. 
 \end{Remark}
 \par
 To a solution $(C_t)_{t\in[0,\tau)}$ of \eqref{eq_sto}, as in the deterministic situation, associate 
 \bq
 \fo t\in[0,\tau),\,\fo u\in\TT,\qquad v(t,u) &\df& \vert \partial_u C(t,u) \vert\eq
  and the arc-length parametrization  $\partial_s\df  \frac{1}{v} \partial_u $ (equivalently  $ ds = v du $).
 \par
 For any $t\in[0,\tau)$ and $u\in\TT$, 
 $T_t (u) $ will stand for the unit tangent vector of the curve $C_t$ at $u$ (i.e.\ in $\mathbb{R}^2$ we have $\mathcal{R} (T ) = \nu$,  where $\mathcal{R} $ is the rotation of angle $ - \frac{\pi}{2} $).
 \par

 The evolution of these objects is dictated by the following result, for $C_0$ regular enough, say $C^{4+ \alpha}$ for the last equation.

\begin{lemma} \label{eq_courbure_sto} Let $(C_t)_{t\geq 0}$ be a solution of \eqref{eq_sto}. The following equations hold in the $(t,s)$-domain of validity:
 \begin{equation} \label{eq_courbure} \left \{ \begin{array}{lcl}
   d_t v_t &=& v_t \rho_t  \left( (- \rho_t + 2h_t) dt + \sqrt{2} dB_t   \right)\\
    \big[ d_t , \partial_s    \big] &= &\rho_t  \left( ( 3 \rho_t - 2h_t) dt - \sqrt{2} dB_t   \right) \partial_s   -  \sqrt{2} \rho_t  dB_t  \partial_s d_t  \\
    d_t T_t &=&- \frac{1}{v_t}(\partial_u \rho_t) \nu_t dt\\
   d_t \rho_t(s) &= & (\partial^2_s \rho_t) dt + \rho_t^2( (3\rho_t - 2h_t)dt - \sqrt{2} dB_t)\\
    \end{array} \right. \end{equation}
 \end{lemma}

\begin{proof}
Since $ C(t,u)$ satisfies $$ d_t C(t,u)  =  (-\rho_t(C(t,u)) + 2h(D(t)) )\nu_{C(t,u)}dt + \sqrt{2} \nu_{C(t,u)} dB_t $$ we have, after differentiation,
cf.\ remark 3.2,
and  shortening  the notation:
$$  d_t \partial_u C  =  ( - \partial_u \rho_t) \nu_t dt +  \left( (-\rho_t + 2h_t )dt + \sqrt{2}  dB_t  \right) \partial_u \nu_t, $$

Also we have $$ \partial_u \nu_t = v \partial_s \nu_t = v \rho_t T_t, $$
so that
$$  d_t \partial_u C  =  ( - \partial_u \rho_t) \nu_t dt + v_t \rho_t \left( (-\rho_t + 2h_t )dt + \sqrt{2}  dB_t  \right)  T_t. $$
Hence we have the following equation:
$$
\begin{aligned}
   d_t (v_t)^2 &= d_t  \vert \partial_u C(t,u) \vert^2 \\ &= 2 \langle d_t \partial_u C(t,u) , \partial_u C(t,u) \rangle +  \langle d_t \partial_u C(t,u),d_t \partial_u C(t,u) \rangle  \\
   & =  2 v_t^2 \rho_t  \left((-\rho_t + 2h_t )dt + \sqrt{2}  dB_t  \right) +  2 v_t^2 \rho_t^2 dt  \\
   &= 2 v_t^2 \rho_t  \left(( 2h_t )dt + \sqrt{2}  dB_t  \right) .\\
\end{aligned}$$

Also  $$d v_t^2 = 2v_t dv_t + dv_t dv_t, $$
where the semi-martingale bracket notation $\langle dv_t,dv_t\rangle$ has been simplified into $dv_t dv_t$. 

 Hence  $$  2v_t dv_t + dv_t dv_t  = 2 v_t^2 \rho_t  \left(2h_t dt + \sqrt{2} dB_t  \right) ,$$
so the Doob-Meyer decomposition of $ v_t$ is  $  dv_t  = \sqrt{2} v_t \rho_t dB_t + dA_t $  where $ A_t$ is a process with finite variation. After identification we find:
 $$ dA_t=  v_t\rho_t  (-\rho_t + 2h) dt $$ and so 
 $$ d_t v_t = v_t \rho_t  \left( (- \rho_t + 2h) dt + \sqrt{2} dB_t   \right) . $$

For the second equation let us observe that for a vector-valued process $X_t $:
$$ \begin{aligned}
d_t \partial_s X_t &= d_t \left(\frac{1}{v_t} \partial _u X_t \right) \\
&= d_t \left(\frac{1}{v_t} \right) \partial_u X_t +  \frac{1}{v_t} \partial_u d_t X_t +
d_t \left(\frac{1}{v_t}\right) d_t (\partial_u  X_t )\\
&=  \frac{\rho_t}{v_t}  \left( ( 3 \rho_t - 2h) dt - \sqrt{2} dB_t   \right) \partial_u X_t + \partial_s d_t X_t -  \sqrt{2} \frac{\rho_t}{v_t}  dB_t  \partial_u d_t X_t  \\
&=  \rho_t  \left( ( 3 \rho_t - 2h) dt - \sqrt{2} dB_t   \right) \partial_s X_t +\partial_s d_t X_t -  \sqrt{2} \rho_t  dB_t  \partial_s d_t X_t . \\
\end{aligned} $$
In other words, we have: 
  $$ \big[ d_t , \partial_s    \big] = \rho_t  \left( ( 3 \rho_t - 2h) dt - \sqrt{2} dB_t   \right) \partial_s   -  \sqrt{2} \rho_t  dB_t  \partial_s d_t  $$
For the third point, let us compute: 
$$
\begin{aligned}
d_t T_t  &= d_t  \left( \frac{1}{v_t}\partial_u C_t \right) = d_t (\frac{1}{v_t}) \partial_u C_t  + \frac{1}{v_t} d_t\partial_u C_t + d_t (\frac{1}{v_t}) d_t\partial_u C_t \\
&= - \frac{\rho_t}{v_t}  \left( (- 3 \rho_t + 2h) dt + \sqrt{2} dB_t   \right) v_t T_t \\
&+ \frac{1}{v_t}  \left(  ( - \partial_u \rho_t) \nu_t dt + v_t \rho_t \left( (-\rho_t + 2h_t )dt + \sqrt{2}  dB_t  \right)  T_t  \right)\\ 
&-2\rho_t^2T_t dt \\
&=- \frac{1}{v_t}(\partial_u \rho_t) \nu_t dt  \\
\end{aligned}
$$ 
that is $ d_t \partial_s C(t,s) = - (\partial_s \rho_t) \nu_t dt $. This equation is equivalent to $$d_t \nu_t = d_t \mathcal{R} (T_t) = (\partial_s \rho_t) T_t dt  .$$ So the processes $T_t $ and $\nu_t$ have finite variation.

For the last point, the curvature $ \rho_t $ satisfies:
$$
\begin{aligned}
 d_t \rho_t &= - d_t \langle \partial_s T_t ,  \nu_t \rangle \\
    &= -  \langle  d_t \partial_s T_t ,  \nu_t \rangle  -\langle \partial_s T_t ,    d_t \nu_t \rangle  \\
    &=  -  \langle  d_t \partial_s T_t ,  \nu_t \rangle . \\
\end{aligned}
$$
In the last line we used that $\partial_s T_t $ is in the normal direction.
Let us compute, using the commutation formula in the first term in the above bracket and the fact that $ \partial_sC_t $ has finite variation:
$$\begin{aligned}
d_t \partial_s T_t &=  d_t  \left( \partial_s  \partial_s C_t  \right)\\
&= \rho_t  \left( ( 3 \rho_t - 2h) dt - \sqrt{2} dB_t   \right) \partial_s  \partial_s C_t + \partial_s  d_t \partial_s C_t - \sqrt{2} \rho_t  dB_t  \partial_s d_t  \partial_s C_t \\
&= \rho_t  \left( ( 3 \rho_t - 2h) dt - \sqrt{2} dB_t   \right) \partial_s  \partial_s C_t+   \partial_s  d_t \partial_s C_t\\
&=\rho_t  \left( ( 3 \rho_t - 2h) dt - \sqrt{2} dB_t   \right) \partial_s  \partial_s C_t+ \partial_s  \left( - (\partial_s \rho_t) \nu_t dt  \right)\\
&= -\rho_t^2  \left( ( 3 \rho_t - 2h) dt - \sqrt{2} dB_t   \right)\nu_t +   \left( - (\partial^2_s \rho_t) \nu_t   - \rho_t (\partial_s \rho_t) T_t  \right)dt.\\
\end{aligned} 
$$

Hence 
\bqn{this}
\nonumber d_t \rho_t &= & -  \langle  -\rho_t^2  \left( ( 3 \rho_t - 2h) dt - \sqrt{2} dB_t   \right)\nu_t +   \left( - (\partial^2_s \rho_t) \nu_t   - \rho_t (\partial_s \rho_t) T_t  \right)dt   ,  \nu_t \rangle  \\
&= & \partial^2_s \rho_t dt + \rho_t^2  \left( ( 3 \rho_t - 2h) dt - \sqrt{2} dB_t  \right).
\eqn
\end{proof}

  \begin{Remark}
We want to stress that \eqref{this} is a SPDE with mobile barrier, since the parameter $s$ lives in the time dependent  interval $[0, \s_t] $.
\end{Remark}

\begin{Remark} \label{retrouve}
After integration, we recover the equation satisfied by $\s_t $ and $ \lambda_t$ obtained by a diffusion generator technique (respectively in Proposition 57 and Equation (106) in \cite{zbMATH07470497})   in the case of a simple curve: 
$$
\begin{aligned}
d_t \s_t &= d_t \int_{0}^{2\pi} v_t(u) \, du\\
           &=  \int_{0}^{2\pi} v_t \rho_t  \left( (- \rho_t + 2h) dt + \sqrt{2} dB_t   \right) \, du\\ 
           &= \lt( -\int \rho_t^2 ds + 2h_t \int \rho_t ds\rt) dt + \left(\int \rho_t \,ds\right)\sqrt{2} dB_t \\
            &= -\lt( \int \rho_t^2 ds\rt)dt + 4\pi h_t   dt + 2\sqrt{2} \pi dB_t, \\
\end{aligned}
$$

 Using, on one hand Stokes theorem, namely  $\int_D \dive(b)\, d\lambda=\int_{\partial D} \langle \nu, b\rangle\, ds$, here in dimension 2 and with $b$ the identity vector field, and on the other hand the computation  done in the proof of Lemma \ref{eq_courbure_sto}, we get:

$$
\begin{aligned}
&d_t \lambda_t\\ &= d_t \frac12 \int_{0}^{2\pi}  \langle C(t,u) , \mathcal{R} (\partial_u C(t,u)) \rangle \, du\\
           &= \frac12 \bigg( \int_{0}^{2\pi} v_t  \left( (- \rho_t + 2h) dt + \sqrt{2} dB_t   \right) \, du  +   \int_{0}^{2\pi}  \langle C(t,u) ,(\partial_u\rho_t(u) )T dt \rangle \, du \\
           &+  \int_{0}^{2\pi} v_t \rho_t \left( (- \rho_t + 2h) dt + \sqrt{2} dB_t   \right)\langle C(t,u) , \nu  \rangle \, du  + 2\int_{0}^{2\pi}  v_t \rho_t dt  \,du  \bigg).\\
\end{aligned}
$$
\par
For the second term in the right hand side we integrate by part and we use $ \partial_u T = -v_t\rho_t \nu_t$ to get:
$$  \int_{0}^{2\pi}  \langle C(t,u) ,(\partial_u\rho_t(u) )T dt \rangle \, du = - \int_0^{2\pi} v_t \rho_t dt \, du - \rho^2_t v_t  \langle C(t,u) , \nu_t \,dt \rangle \, du . $$
And then 
$$
\begin{aligned}
d_t \lambda_t &=  \frac12 \bigg( \int_{0}^{2\pi} v_t  (2h dt + \sqrt{2} dB_t  ) \, du  \\
           &+  \int_{0}^{2\pi} v_t \rho_t  (  2h dt + \sqrt{2} dB_t )  \langle C(t,u) , \nu \rangle \, du    \bigg).
\end{aligned}
$$
In the last term of the right hand side  we can integrate by part again to get 
$$
\begin{aligned} 
&\int_{0}^{2\pi} v_t \rho_t \langle C(t,u) , \nu \rangle \, du  = - \int_{0}^{2\pi}  \langle C(t,u) , \partial_u T \rangle \, du \\ 
&=  \int_{0}^{2\pi} (-\partial_u \langle C(t,u) ,  T \rangle   +v_t  )\, du =  \int_{0}^{2\pi} v_t  \, du.
\end{aligned} $$
Hence:
$$
\begin{aligned}
d_t \lambda_t &=    \int_{0}^{2\pi} v_t  (2h dt + \sqrt{2} dB_t  ) \, du  \\
            &= \frac{2 \s_t^2} {\ld_t} dt + \sqrt{2} \s_t dB_t.  
\end{aligned}
$$
\end{Remark}

Let us extend the observation made in the second paragraph of Section \ref{2.4} to the present stochastic setting.
 Consider the following tangential finite-variation perturbation of equation \eqref{eq_sto}: for any $t\in[0,\tau)$ and $u\in\TT$,
\bqn{eq_sto_mod}
 d_t { C}(t,u) =  \left( -{\rho_t}(u) + 2h({D_t}) )dt + \sqrt{2}  dB_t  \right){\nu}_t(u) + (\aaa_t(u) dt){T}_t (u)  
 \eqn
 for quantities $\aaa_t(u)$ that will determined later in \eqref{aaa}.
 
 Again the curves have the same shape as those of \eqref{eq_sto}, only the $u$-parametrisation changes, reason why
 we used the same notation $C(t,u)$.
In particular the lifetime $\tau$ of \eqref{eq_sto_mod} coincides with that of \eqref{eq_sto}. 
With computations similar to  those of the proof of Lemma \ref{eq_courbure_sto}, we get:

\begin{lemma} Letting $ (C_t)_{t\in[0,\tau)}$ be a solution of \eqref{eq_sto_mod},  we have:
\begin{enumerate}

\item[(i)] $ \begin{aligned} 
 d_t \partial_u C_t &=  \left( (-\partial_u \rho_t - \rho_t v_t \aaa_t ) dt  \right) \nu_t \\
 &+  \left( v_t\rho_t  [(-\rho_t + 2h )dt + \sqrt{2} dB_t ]    + \partial_u \aaa_t dt \right) T_t .
 \end{aligned}$

\item[(ii)] $ dv_t =  v_t\rho_t   [(-\rho_t + 2h )dt +\sqrt{2} dB_t ]  \  + \partial_u \aaa_t dt .$

\item[(iii)] $ d_t T_t = -\frac{1}{v} (\partial_u \rho_t + \rho_t v_t \aaa_t )dt \nu_t. $
\end{enumerate}

\end{lemma}
\begin{proof}
For the first point, we differentiate term by term and we use  that: 
  $$ \partial_u \nu_t = v_t \rho_t T_t $$ and  so $$  \partial_u T_t = - v_t \rho_t \nu_t .$$ 
  
For the second point:
$$
\begin{aligned}
 d_t (v_t)^2 &= d_t  \vert \partial_u C(t,u) \vert^2 = 2 \langle d_t \partial_u C(t,u) , \partial_u C(t,u) \rangle +  \langle d_t \partial_u C(t,u),d_t \partial_u C(t,u) \rangle  \\
   &=  2 v_t  \left( v_t \rho_t  \left( (-\rho_t + 2h_t )dt + \sqrt{2}  dB_t  \right) + \partial_u \aaa_t dt  \right) + 2  v^2_t\rho_t^2 dt.\\
\end{aligned}$$ 
Then we  write $ dv_t^2 = 2v_t dv_t + dv_tdv_t  $, and we identify the martingale part and the finite variation part of $v_t$ to get the conclusion.

For the last point we compute:

$$
\begin{aligned}
d_t T_t  &= d_t  \left( \frac{1}{v_t}\partial_u C_t \right) = d_t \left(\frac{1}{v_t}\right) \partial_u C_t  + \frac{1}{v_t} d_t\partial_u C_t + d_t \left(\frac{1}{v_t}\right) d_t\partial_u C_t \\
&=  \left( - \frac{\rho_t}{v_t}  \left( (- 3 \rho_t + 2h) dt + \sqrt{2} dB_t  \right) - \frac{\partial_u \aaa_t}{v_t^2} dt  \right) v_t T_t \\
&+ \frac{1}{v_t}  \left(  ( - \partial_u \rho_t -\rho_tv_t\aaa_t) \nu_t dt +  \left( v_t \rho_t ( (-\rho_t + 2h_t )dt + \sqrt{2}  dB_t ) + \partial_u\aaa_t dt  \right)T_t  \right)\\ 
&-2\rho_t^2T_t dt \\
&= - \frac{1}{v_t}(\partial_u \rho_t + \rho_tv_t\aaa_t)  dt \nu_t .
\end{aligned}
$$ 
\end{proof}

In the above lemma, if the curvature is positive and  if we take 
\bqn{aaa} \aaa_t &= &\frac{-\partial_u \rho_t}{v_t \rho_t}\eqn
we get that $T_t$ and $\nu_t $ become constant in time hence the angle $\theta $ becomes constant in time, as desired.
The following lemma  describes the evolution of the curvature in this system of coordinates. Let us first reinforce Assumption \ref{hypoC0}:
\begin{hypo} \label{hypoC02} 
\CPK{ The initial curve $ C_0\st \mathbb{T}\ni u\mapsto C_0(u) \in\RR^2 $ is  simple, closed, strictly convex  and $C^{4 + \alpha} $, for some $\alpha>0$.}
 \end{hypo}
\begin{lemma} \label{eq_rho_theta}  Assuming Hypothesis \ref{hypoC02}, the solution to 
$$ d_t\rho_t(\theta) = \rho^2_t(\theta) (\partial^2_{\theta} \rho_t(\theta)) dt + \rho^2_t(\theta)   \left( (3\rho_t(\theta) - 2h )dt - \sqrt{2} dB_t  \right), $$
is well-defined  for all 
 $ 0 \le t < \tau_0 \wedge  \tau $, where  $ \tau_0 = \inf \{t \ge 0\st \exists u \in [0 , 1] , \rho_t(u) = 0 \}$ and $\tau $ is the lifetime of \eqref{eq_sto_mod}. Due to  Hypothesis \ref{hypoC02}, all the quantities $ h, \rho, \partial_{\te}\rho, \partial^2_{\te}\rho $ are bounded until $\tau $.
 \end{lemma}

\begin{proof}
By the above choice of $ \alpha$ we have:
$$
\begin{aligned}
 0 =\partial_u d_t T_t &= d_t (\partial_u T_t) \\
            &=d_t (-\rho_t v_t \nu_t ).
\end{aligned} $$ 
Since $ \nu_t $ is constant in time we get:
$$ d_t (v_t \rho_t )=0 ,$$
and  so
$$
\begin{aligned}
 0 =d_t (v_t \rho_t ) &= d_t (v_t) \rho_t + v_t d_t \rho_t + dv_t d \rho_t .\\
\end{aligned} $$ 

We get that $\rho_t $ satisfies the following stochastic differential equation:
$$
\begin{aligned}
  v_t d\rho_t &= - \rho_t d_t (v_t) -  dv_t d\rho_t \\
              &=  - \rho_t   \left( v_t\rho_t   [(-\rho_t + 2h )dt +\sqrt{2} dB_t ]  + \partial_u \aaa_t dt    \right) -   dv_t d\rho_t .  \\
\end{aligned} $$ 

After identification, the martingale part of $d\rho_t$ is $ - \sqrt{2} \rho^2_t dB_t,$
 hence by this choice of $\alpha $
$$
\begin{aligned}
   d\rho_t  &=  - \rho^2_t   \left(   (-\rho_t + 2h )dt + \sqrt{2} dB_t  \right)    -\frac{\rho_t}{v_t} \partial_u \aaa_t dt +2 \rho^3_t dt \\
   &=  - \rho^2_t   \left(   (-3\rho_t + 2h )dt + \sqrt{2} dB_t  \right)    -\frac{\rho_t}{v_t} \partial_u \aaa_t dt  \\
   &= -\frac{\rho_t}{v_t} \partial_u  \left(\frac{-\partial_u \rho_t}{v_t \rho_t}\right)  dt- \rho^2_t   \left(   (-3\rho_t + 2h )dt + \sqrt{2} dB_t  \right).
\end{aligned} $$
 
 Recall that $ T=(\cos (\theta),\sin (\theta) ) $ and so $ \partial_{\theta} T = - \nu $. So by the chain rule we have:
 $$ \begin{aligned}
- \nu = \frac{\partial T}{\partial \theta} &= \frac{\partial u}{\partial \theta}\frac{\partial T}{\partial u} 
= \frac{\partial u}{\partial \theta} (-v \rho ) \nu . 
 \end{aligned} $$
 Hence 
 \bqn{vrho}
 \frac{\partial u}{\partial \theta}\ =\ \frac{1}{v \rho} &\hbox{and}& \partial_{\theta}\ =\ \frac{1}{v\rho} \partial_u.\eqn\par The previous evolution equation of $ \rho_t $ becomes
$$
\begin{aligned}
   d\rho_t  &= 
   \frac{\rho_t}{v_t} \partial_u \partial_{\theta}   \rho_t  dt- \rho^2_t   \left(   (-3\rho_t + 2h )dt + \sqrt{2} dB_t  \right)\\
   &= \rho^2_t \partial^2_{\theta} \rho_t dt +  \rho^2_t   \left(   (3\rho_t - 2h )dt - \sqrt{2} dB_t  \right).
\end{aligned} $$ 
 \end{proof}
\begin{thm} \label{th_cores}
Assume Hypothesis \ref{hypoC02}, in particular $\rho_0>0$.

 Let $\rho_t(\theta) $ be a solution of the following elliptic partial stochastic differential equation:

\begin{equation} \left \{  \begin{array}{lcl}\label{eq_rho}
 d_t \rho_t(\theta) & = & \rho^2_t(\theta) (\partial^2_{\theta} \rho_t )dt + \rho^2_t(\theta)   \left( (3\rho_t(\theta) - 2h )dt - \sqrt{2} dB_t  \right) \\
  \rho_0(\theta) &= & \rho_0(\theta). \\
 \end{array}  \right.
\end{equation}
with lifetime $  \tau_2$,  namely the solution has to be regular up to order $2$ for at least all times smaller than $ \tau_2$.

Then $\tau_0 \wedge \tau =  \tau_2 $, and for all $ t<  \tau_2 $, we have $\rho_t(\theta)>0$ for all $\theta\in \TT$.  Moreover the solution to \eqref{eq_rho} is unique and it provides  the solution of \eqref{eq_sto}  through:
$${C}(t,\te)\df  \tilde{C}(t,\te) + \int _0^t (-\partial_{\te}\rho_u(0), \rho_u(0)-2h_u ) \,du - (0, \sqrt{2} B_t  ) $$
where
$$ \tilde{C}(t,\theta) =  \left( \int_0^{\theta} \frac{\cos(\theta_1)}{\rho_t(\theta_1)} \, d \theta_1 ,  \int_0^{\theta} \frac{\sin(\theta_1)}{\rho_t(\theta_1)} \, d \theta_1   \right). $$
\par

\end{thm}

\begin{proof}
 By  lemma \ref{eq_rho_theta}, \eqref{eq_rho} admits a solution and $\tau_0 \wedge \tau \le  \tau_2 $.
Note that the quantity~$h_t$ could be expressed in terms of $ \s_t$ and $ \ld_t$ and these quantities also depend on the integral of $ \rho  $ as seen in Remark \ref{rm_sys}  below and so $h $ is bounded until~$ \tau_2 $. 
\\
From \eqref{eq_rho}, we get for all $ t <  \tau_2$,
$$ {\rho}_t(\te) = \rho_0(\te)\exp^{ \int_0^t -\sqrt{2} \rho_s (\te)dB_s +  \left(\rho_s (\te)\partial^2_{\te}\rho_s (\te) + 2 \rho_s (\te) (\rho_s (\te) - h_s)  \right) ds }$$
which is positive, yielding $\tau_2\le \tau_0$.\par

 Recall  Lemma 4.1.1 in \cite{Gage_Hamilton}, or see the beginning of Section \ref{conv-sym}, that says a $2 \pi $ periodic positive function $\rho $ represents the curvature of a simple closed strictly convex plane curve if and only if \eqref{cores} is satisfied.

 Here this equation is satisfied by $ \rho_0(\te) $.
 So we have to check that this relation is preserved over time for $ \rho_t(\te) $ solution of \eqref{eq_rho}.
 We will only verify this fact for the first coordinate, the computation will be the same for the second one.
 Using It\^o calculus we get for $0 \le t <  \tau_2 $:
\bqn{d1r}
\nonumber d_t \frac{1}{\rho_t} &=& -\frac{1}{\rho^2_t} d\rho_t +  \frac{1}{\rho_t^3} d\rho_td\rho_t \\
\nonumber   &=& -  (\partial^2_{\theta} \rho_t(\theta))dt -  \left( (3\rho_t(\theta) - 2h )dt - \sqrt{2} dB_t  \right) +2 \rho_t dt \\
  &=& -(\partial^2_{\theta} \rho_t(\theta))dt -  \left( (\rho_t(\theta) - 2h )dt - \sqrt{2} dB_t  \right) .
\eqn\par
 And so after integration by part we get for $0 \le t <  \tau_2 $:
 $$
 \begin{aligned}
 d_t \int_0^{2\pi} \frac{\cos(\te)}{\rho_t(\te)} \,d\te &= \int_0^{2\pi} d_t\frac{\cos(\te)}{\rho_t(\te)} \,d\te  \\
 &=\int_0^{2\pi} \cos(\te)  \left( -(\partial^2_{\theta} \rho_t(\theta))dt -  \left( (\rho_t(\theta) - 2h )dt - \sqrt{2} dB_t  \right)  \right)  \,d\te  \\
 &= -\bigg( \int_0^{2\pi} \cos(\te)  \left( \partial^2_{\theta} \rho_t(\theta) + \rho_t(\theta)     \right)  \,d\te \bigg)dt \\
 &=0 .
 \end{aligned}
 $$
 We get that, for all  $0 \le t <   \tau_2 $, $ \rho_t $ is the curvature of a  simple closed strictly convex plane curve.
 Let us write the curve as:   
 $$ \tilde{C}(t,\theta) =  \left( \int_0^{\theta} \frac{\cos(\theta_1)}{\rho_t(\theta_1)} \, d \theta_1 ,  \int_0^{\theta} \frac{\sin(\theta_1)}{\rho_t(\theta_1)} \, d \theta_1   \right) .$$
 
 We only have to check that $({C}(t,\theta))_\theta $ solves Equation \eqref{eq_sto} up to some tangential component.

 We have: 
 $$ \begin{aligned}
 d_t {C}(t,\te)  &= d_t  \left( \int_0^{\theta} \frac{\cos(\theta_1)}{\rho_t(\theta_1)} \, d \theta_1 ,  \int_0^{\theta} \frac{\sin(\theta_1)}{\rho_t(\theta_1)} \, d \theta_1   \right)  \\
 &+ (-\partial_{\te}\rho_t(0)dt, (\rho_t(0)- 2h_t)dt - \sqrt{2} dB_t  )\\
 &= \bigg( \int_0^{\te} \cos(\te_1)  \left( -(\partial^2_{\theta_1} \rho_t(\theta_1))dt -  \left( (\rho_t(\theta_1) - 2h )dt - \sqrt{2} dB_t  \right)  \right)  \,d\te_1 , \\
 & \int_0^{\te} \sin(\te_1)  \left( -(\partial^2_{\theta_1} \rho_t(\theta_1))dt -  \left( (\rho_t(\theta_1) - 2h )dt - \sqrt{2} dB_t  \right)  \right)  \,d\te_1  \bigg) \\
 & + (-\partial_{\te}\rho_t(0)dt, (\rho_t(0)- 2h_t)dt - \sqrt{2} dB_t  ).
 \end{aligned} $$

After two integrations by parts, we have  for the first term in the right hand side:
$$ \begin{aligned}
&\int_0^{\te} \cos(\te)  \left( -(\partial^2_{\theta_1} \rho_t(\theta_1))dt -  \left( (\rho_t(\theta_1) - 2h )dt - \sqrt{2} dB_t  \right)  \right)  \,d\te_1  \\
&=-\bigg\{ [\cos(\te_1) \partial_{\te_1} \rho_t ]_0^{\te} dt+ [\sin(\te_1)  \rho_t ]_0^{\te}dt + [\sin(\te)]  \left( -2h dt - \sqrt{2} dB_t \right) \bigg\} \\
&= - \cos(\te) \partial_{\te} \rho_t(\te)dt +  \partial_{\te} \rho_t(0)dt - \sin(\te)  \left( (\rho_t-2h) dt - \sqrt{2} dB_t \right) .
 \end{aligned}
 $$
  For the second term, we have:
$$ \begin{aligned}
&\int_0^{\te} \sin(\te)  \left( -(\partial^2_{\theta_1} \rho_t(\theta_1))dt -  \left( (\rho_t(\theta_1) - 2h )dt - \sqrt{2} dB_t  \right)  \right)  \,d\te_1  \\
&=-\bigg\{ [\sin(\te_1) \partial_{\te_1} \rho_t ]_0^{\te}dt - [\cos(\te_1)  \rho_t ]_0^{\te}dt - [\cos(\te)]_0^{\te}  \left( -2h dt - \sqrt{2} dB_t \right) \bigg\} \\
&= - \sin(\te) \partial_{\te} \rho_t(\te) dt + \cos(\te)  \left( (\rho_t-2h) dt - \sqrt{2} dB_t \right) \\
&-  \left( (\rho_t(0)-2h) dt - \sqrt{2} dB_t \right)  .
 \end{aligned}
 $$  
 Hence:
 $$ \begin{aligned}
 d_t {C}(t,\te)  &=  ((-\rho_t + 2h ) dt + \sqrt{2}  dB_t )\nu  - (\partial_{\te} \rho_t dt ) T\\
 \end{aligned} .$$

This is  \eqref{eq_sto_mod}  and so up a parametrization, this is a solution to \eqref{eq_sto}. 
Since a solution to \eqref{eq_rho} produces a solution to \eqref{eq_sto}, by uniqueness of solution to  \eqref{eq_sto}, we get the uniqueness of the solution to \eqref{eq_rho}, and 
$ \tau_2 \le \tau $. So we proved that $\tau_2=\tau\wedge\tau_0$.
 
 \end{proof}

 We will show that Equation \eqref{eq_sto} preserves the positivity of the curvature.

\begin{lemma}\label{conserve_convex}
Assume Hypothesis \ref{hypoC02}
and consider the solution to \eqref{eq_sto}. We have  $\rho_t > 0$ for all $ t < \tau$, where $\tau  $ is any lifetime of \eqref{eq_sto}, moreover $\tau =  \tau_2 $.

\end{lemma}
\begin{proof}

 Suppose that  $ \tau_0  < \tau $, so $h_t, \rho_t(\te), \partial_{\te}\rho_t(\te) , \partial_{\te}^2\rho_t(\te)  $ are bounded for all  $ t \le \tau_0$, and
$$ \rho_{\tau_0}(\te) = \rho_0(\te)\exp^{ \int_0^{\tau_0} -\sqrt{2} \rho_s (\te)dB_s +  \left(\rho_s (\te)\partial^2_{\te}\rho_s (\te) + 2 \rho_s (\te) (\rho_s (\te) - h_s)  \right) ds },$$
and we get a contradiction.  By Theorem \ref{th_cores},  we get $  \tau =  \tau_2  .$
\end{proof}

\begin{Remark} \label{rm_sys}
Let us compute the equation satisfied by $h$ when we know the equation of $ \rho$. Resorting to \eqref{d1r} and recalling from \eqref{vrho}
that $1/(v\rho)=\partial u/\partial \theta=1$,
we get by Stokes Theorem:
$$\begin{aligned}
 d \s_t &= d  \int_0^{2\pi} \vert  \partial_{\te} C(t,\te)  \vert \, d\te \\
            &= d  \int_0^{2\pi} \frac{1}{\rho_t(\te)} \, d\te \\
        &= \int_0^{2\pi}  \left(  - \partial^2_{\te}\rho_t(\te) dt  - ( (\rho_t(\te) - 2h )dt - \sqrt{2} dB_t )  \right)   \, d\te \\
        &= \lt( -\int_0^{2\pi}   \rho_t(\te)   \, d\te \rt) dt  + 4 \pi h_t dt +2\sqrt{2} \pi dB_t \\
        &= \lt(-\int_0^{\s_t}   \rho_t^2(s) ds \rt) dt + 4 \pi h_t dt +2\sqrt{2} \pi dB_t .
\end{aligned} $$
\par
By a similar computation as above, we also have:

$$\begin{aligned}
 d \ld_t &= d \frac12  \int_0^{2\pi} \langle C(t,\te), \nu_t(\te)  \rangle \frac{d\te}{\rho_t(\te)} \\
 &= \frac12 \bigg\{ \int_0^{2\pi} \Bigg( \langle dC(t,\te), \nu_t(\te)  \rangle \frac{1}{\rho_t(\te)} +  \langle C(t,\te), \nu_t(\te)  \rangle d (\frac{1}{\rho_t(\te)} ) + \\
 &+  \langle d C(t,\te), \nu_t(\te)   \rangle d (\frac{1}{\rho_t(\te)} )\Bigg)  \,d\te \bigg\} \\
 &= \frac12 \bigg\{ \int_0^{2\pi} \Bigg(  \left(  (-\rho_t(\te) +2h )dt + \sqrt{2} dB_t       \right)  \frac{1}{\rho_t(\te)}     \\
  &+  \langle C(t,\te), \nu_t(\te)  \rangle  \left(  - \partial^2_{\te}\rho_t(\te)dt  - ( (\rho_t(\te) - 2h )dt - \sqrt{2} dB_t )  \right) +  2 dt  \Bigg) \,d\te \bigg\} 
\end{aligned}
$$ 
After integrating by part two times and using $\pa_\theta \nu=T$, we get:
$$ \begin{aligned}
&\int_0^{2\pi} - \langle C(t,\te), \nu_t(\te)  \rangle \partial^2_{\te}\rho_t(\te)  \, d\te  = \int_0^{2\pi} \partial_{\te} ( \langle C(t,\te), \nu_t(\te)  \rangle) \partial_{\te}\rho_t(\te)  \, d\te \\
&=  \int_0^{2\pi} \partial_{\te}\rho_t(\te)    \langle C(t,\te), T_t(\te) \rangle \, d\te \\
&  = - \int_0^{2\pi} \rho_t(\te) \lt( \frac{1}{\rho_t(\te)}  -  \langle C(t,\te), \nu_t(\te)\rangle \rt)   \, d\te . 
\end{aligned}
$$
\par
Taking into account that $\partial_{\te} T = -\nu $, we have
$$ \begin{aligned}
& \int_0^{2\pi}  \langle C(t,\te),  \nu_t(\te) \rangle \,d\te =  -\int_0^{2\pi} \langle C(t,\te), \partial_{\te}T_t(\te) \rangle  \,d\te \\
 &=  -\int_0^{2\pi} \partial_{\te} ( \langle C(t,\te), T_t(\te) \rangle ) - \frac{1}{\rho_t(\te)}  \,d\te = \int_0^{2\pi}  \frac{1}{\rho_t(\te)}  \,d\te .
\end{aligned} $$   
Putting the two computations above in the evolution equation of $ \ld_t$ we get:

$$ \begin{aligned}
d\ld_t  &=  \frac12 \bigg\{ \int_0^{2\pi} \Bigg(  \left( 2h dt + \sqrt{2} dB_t       \right)  \frac{1}{\rho_t(\te)}    \\&- \langle C(t,\te), \nu_t(\te)  \rangle  \left(   - 2h dt - \sqrt{2} dB_t )  \right)  \Bigg) \,d\te \bigg\} \\
 &=  \int_0^{2\pi}  \frac{1}{\rho_t(\te)}  \,d\te    \left( 2h dt + \sqrt{2} dB_t       \right)  =  d \ld_t = \frac{2 \s_t^2}{\ld_t} dt + \sqrt{2} \s_t dB_t .
\end{aligned} $$

 So we have to interpret \eqref{eq_rho} as a system where we have:
 
    \begin{equation}  \left \{ \begin{array}{lcl} \label{eq_geom}
  d \s_t &= & \left( -\int_0^{2\pi}   \rho_t(\te)   \, d\te  \right) dt  + 4 \pi  \frac{ \s_t}{\ld_t}  dt +2\sqrt{2} \pi dB_t \\
  d \ld_t &=& \frac{2 \s_t^2}{\ld_t} dt + \sqrt{2} \s_t dB_t \\
  h_t &=& \frac{\s_t}{\ld_t} 
 \end{array} \right.
\end{equation}
  
\end{Remark}

Using the above theorem and  Lemma  \ref{eq_rho_theta}, we get the following corollary:
\begin{corollary}\label{cor1} Assume Hypothesis \ref{hypoC02}. There is a one to one correspondence  between the solutions of \eqref{eq_sto},  \eqref{eq_rho} and \eqref{eq_sto_mod}.
\end{corollary}
\begin{proof}
Use Theorem \ref{th_cores}, Lemma \ref{eq_rho_theta} and \ref{conserve_convex}. 

\end{proof}

Consider the following stochastic curvature flow,
  \begin{equation} \left \{  \begin{array}{lcl}\label{eq_sto_CF}
 d_t C(t,u) & = & (-\rho_t(C(t,u)) )\nu_{C(t,u)}dt + \sqrt{2} \nu_{C(t,u)} dB_t  \\
  C(0,u) &= & C_0(u) \\
 \end{array}  \right. \end{equation}

\begin{corollary}
Assume Hypothesis \ref{hypoC02} and
consider the solution to \eqref{eq_sto_CF}. We have  $\rho_t > 0$ for all $ t < \tau$, where $\tau  $ is any lifetime of \eqref{eq_sto_CF}.
\end{corollary}
\begin{proof}
With similar computation as in the above lemma, and since $\rho_0 >0$, we have:

 \begin{equation} \left \{  \begin{array}{lcl}\label{eq_sto_rho}
  d_t\rho_t(\theta) = \rho^2_t(\theta) (\partial^2_{\theta} \rho_t(\theta)) dt + \rho^2_t(\theta)   \left( (3\rho_t(\theta))dt - \sqrt{2} dB_t  \right), \\
  \rho_0 = \rho_0
 \end{array}  \right. \end{equation}
 and the proof is similar to the proof of Lemma \ref{eq_rho_theta}, Theorem \ref{th_cores} and Lemma \ref{conserve_convex}.  
  
\end{proof}

\begin{corollary}\label{cor2} Assume Hypothesis \ref{hypoC02}, there is a one to one correspondence  between the solutions of \eqref{eq_sto_CF} and \eqref{eq_sto_rho}.
\end{corollary}
\begin{proof}
The proof is similar to that of Theorem \ref{th_cores}, just remove all the $h_t$.
\end{proof}

\section{Long time existence}
\subsection{Evolution of geometric quantities along the stochastic flow  \eqref{eq_sto}}


 \begin{proposition} \label{prop-ev}\CPK{Assume Hypothesis \ref{hypoC02}.}
 Let $(C_t)_{t\in[0,\tau)}$  be the solution of \eqref{eq_sto}. For any $t\in[0,\tau)$, denote $\lambda_t $ the volume of $D_t$ and  $\s_t $ the perimeter of $C_t$.  We have the following equations for  $t\in[0,\tau)$ (with our usual notational shortcuts):
\begin{enumerate}[i)]
 \item $\displaystyle  d_t ( \s^2_t -4\pi \ld_t ) \le -2 \pi  \left( \frac{\s^2_t - 4\pi \ld_t }{ \ld_t }  \right) dt ,$ \\
 \item $\displaystyle d \frac{1}{\rho_t(\theta)} = 
-\partial^2_{\theta} \rho_t(\theta)dt - (\rho_t(\theta) - 2h )dt + \sqrt{2} dB_t $, \\
\item $\displaystyle  d \int_0^{2 \pi}   \frac{1}{\rho^2_t}\, d\te =  - 2 \int_0^{2 \pi}  \left(  \partial_{\te} \log (\rho_t )  \right)^2 \,  d\te dt + 2 d \ld_t $ ,\\
\item  $$ \begin{aligned}
 d \int_0^{2\pi} \log (\rho_t(\te)) \, d\te
  =& -\int_0^{2\pi}  (\partial_{\theta} \rho_t )^2\, d\te dt + 2 \int_0^{2\pi}   \left(\rho_t(\theta) - \frac{h}{2} \right)^2 \, d\te dt \\
 &- \pi h^2 dt -  \sqrt{2} \int_0^{2\pi}  \rho_t(\theta) \, d\te  dB_t  .\\
\end{aligned} $$ 

\end{enumerate} 
 \end{proposition}
\begin{proof}

For equation i): using equation  \eqref{eq_geom} and It\^o formula we have 

\begin{align*}
 d( \s^2_t -4\pi \ld_t  ) &= 2\s_t d\s_t +  d\s_t d\s_t -4\pi d\ld_t\\
 &= 2\s_t  \left( -\int_0^{2\pi}   \rho_t(\te)   \, d\te  dt  + 4 \pi  \frac{ \s_t}{\ld_t}  dt +2\sqrt{2} \pi dB_t    \right) + 8\pi^2 dt  \\
 & -  \frac{8\pi \s_t^2}{\ld_t} dt  - 4 \pi \sqrt{2} \s_t dB_t \\
 &= 2\s_t  \left( -\int_0^{2\pi}   \rho_t(\te)   \, d\te      \right) dt + 8\pi^2 dt  \\
 &\le \left( - 2\pi \frac{\s^2_t}{\ld_t} + 8\pi^2\right) dt\\
 & = -2 \pi  \left( \frac{\s^2_t - 4\pi \ld_t }{ \ld_t }  \right) dt \\
 \le 0 \\
\end{align*}
where we use  the preservation of the convexity along the flow (Lemma \ref{conserve_convex}) and Gage inequality  for convex curve \cite{Gage_conv}:  
\bqn{Gage} \pi h(D)& = &\pi \frac{\s (C)}{\ld(D)}\ \le\ \int_C \rho^2(s) ds\ =\ \int _0^{2\pi} \rho(\te)  d \te. \eqn Also in the last inequality we use the isoperimetric estimate.
So the isoperimetric deficit $ \s^2_t -4\pi \ld_t $ is non-increasing along the flow. One of the geometric meaning of the isoperimetric deficit is the following Bonnesen inequality \cite{Gage_curve}: 
\bq \pi^2 (\rao -\rai)^2& \le& \s^2(\partial D) -4\pi \ld(D) \eq
where $\rai,\rao$ are respectively the inradius and the circumradius of $D$.

For equation ii): it is done in the proof of Theorem \ref{th_cores}.

For equation iii): using It\^o formula in the point (ii) we get

 $$
 \begin{aligned}
 d_t \frac{1}{\rho^2_t} &=  \frac{2}{\rho_t}  \left( -(\partial^2_{\theta} \rho_t(\theta))dt - ( (\rho_t(\theta) - 2h )dt + \sqrt{2} dB_t  \right) + 2 dt \\
 &= - \frac{2}{\rho_t}\partial^2_{\theta} \rho_t(\theta)dt + \frac{4}{\rho_t}hdt + \frac{2 \sqrt{2}}{\rho_t} dB_t.\\
\end{aligned} 
 $$
 Integrating the above equality  we get (since $ \int_ 0^{2\pi} \frac{1}{\rho} \, d\te = \s_t$)
 
 $$
 \begin{aligned}
 d \int_0^{2 \pi}   \frac{1}{\rho^2_t}\, d\te &=  - 2 \int_0^{2 \pi}  \left( \frac{( \partial_{\te} \rho_t )}{\rho_t}  \right)^2 \,  d\te dt + 4\frac{\s^2_t}{\ld_t} dt +  2 \sqrt{2} \s_t dB_t \\
  &=  - 2 \int_0^{2 \pi}  \left(  \partial_{\te} \log (\rho_t )  \right)^2 \,  d\te dt + 2 d \ld_t .\\
 \end{aligned} $$

For equation iv) we use \eqref{eq_rho} and It\^o formula:

$$ \begin{aligned}
 d \log (\rho_t(\te)) &= \frac{1}{\rho_t(\te)} d \rho_t(\te) - \frac{1}{2 \rho^2_t(\te)}d \rho_t(\te)d \rho_t(\te)\\
 &= \rho_t(\theta) (\partial^2_{\theta} \rho_t )dt + \rho_t(\theta)   \left( (3\rho_t(\theta) - 2h )dt - \sqrt{2} dB_t  \right) - \rho^2_t(\te) dt \\
 &=\rho_t(\theta) (\partial^2_{\theta} \rho_t )dt + 2\rho_t(\theta)   (\rho_t(\theta) - h )dt - \sqrt{2}   \rho_t(\theta)   dB_t  .\\
\end{aligned} $$
Integrating the above equation, we get:

 $$ \begin{aligned}
 d \int_0^{2\pi} \log (\rho_t(\te)) \, d\te
 =& -\int_0^{2\pi}  (\partial_{\theta} \rho_t )^2\, d\te dt + 2 \int_0^{2\pi} \rho_t(\theta)   (\rho_t(\theta) - h ) \, d\te dt \\
 &-  \sqrt{2} \int_0^{2\pi}  \rho_t(\theta) \, d\te  dB_t  .\\
 =& -\int_0^{2\pi}  (\partial_{\theta} \rho_t )^2\, d\te dt + 2 \int_0^{2\pi}   \left(\rho_t(\theta) - \frac{h}{2} \right)^2 \, d\te dt \\
 &- \pi h^2 dt -  \sqrt{2} \int_0^{2\pi}  \rho_t(\theta) \, d\te  dB_t  .\\
\end{aligned} $$ 

 \end{proof}
 
 \begin{Remark}
 Note that   $\displaystyle \int_0^{2 \pi} \frac{1}{\rho^2_t}\, d\te - 2 \ld_t \ge  \frac{1}{2 \pi} \s^2_t - 2 \ld_t \ge 0 $ where the last bound is the isoperimetric inequality. Hence
 $$ 0 \le  \int_0^{2 \pi} \frac{1}{\rho^2_t}\, d\te - 2 \ld_t = - 2 \int_0^t \int_0^{2 \pi}  \left(  \partial_{\te} \log (\rho_s )  \right)^2 \,  d\te ds + A_0 $$
where $ A_0 = \int_0^{2 \pi} \frac{1}{\rho^2_0}\, d\te - 2 \ld_0 \ge 0$.
So if moreover  $C_0 $ is a curve in the set $\mathcal{S}_n $ of $n$-symmetric convex curves with star shaped skeleton for some $n\ge 2$ (see Section \ref{conv-sym} for the definition) using Proposition \ref{prop_S}, $C_t \in \mathcal{S}_n $ and $\te \mapsto \rho_t(\te) $ is non-decreasing and the above equation gives:
$$ 0 < 2 \ld_t \le \frac{1}{\rho_t(0)} \s_t \le \frac{2\pi}{\rho^2_t(0)} $$ 
so $ 0 < \rho_t(0) \le \sqrt{\frac{\pi}{\ld_t}} $ and $ 0 < \rho_t(0) \le \frac{h_t}{2} . $ On the other hand we have 
$$ 0 <  \int_0^{2 \pi} \frac{1}{\rho^2_t}\, d\te \le  A_0 + 2 \ld_t ,$$
and if $ C_0 \in \mathcal{S}_n $ then 
$$ 0 < \frac{2\pi}{\rho^2_t(\pi / 2)} \le \frac{\s_t}{\rho_t(\pi / 2)}  \le  A_0 + 2 \ld_t $$
so $ \sqrt{ \frac{2 \pi}{ A_0 + 2 \ld_t } } \le \rho_t(\pi / 2) $ and note also by the Gage inequality \eqref{Gage} we  have $ \frac{h_t} {2} \le \rho_t(\pi / 2)  . $
 \end{Remark} 
 
\begin{lemma}\label{prop-h-sm}
 $ (h_t)_{t\in[0,\tau)}$ is a positive super martingale, so it is almost surely bounded on $[0,\tau)$.
 \end{lemma} 
 \begin{proof}
 Using equation  \eqref{eq_geom} and It\^o formula we have
 \begin{align*}
 d h_t &= d \left(\frac{\sigma_t}{\lambda_t}\right) \\
       &= \frac{1}{\lambda_t} d\sigma_t + \sigma_t d\left(\frac{1}{\lambda_t}\right) + d  \sigma_t d\left(\frac{1}{\lambda_t}\right) \\       
      &= \frac{1}{\lambda_t}   \left(  \left( -\int_0^{2\pi}   \rho_t(\te)   \, d\te \right)  dt  + 4 \pi  \frac{ \s_t}{\ld_t}  dt +2\sqrt{2} \pi dB_t  \right)  - \frac{\sqrt{2} \s_t^2}{\ld_t^2} dB_t
      - \frac{4\pi \s_t}{\ld_t^2}dt \\
      &=  -\frac{1}{\lambda_t}\left(\int_0^{2\pi}   \rho_t(\te)   \, d\te\right) dt + \sqrt{2} \left( \frac{2 \pi \ld_t - \s_t^2 }{\lambda_t^2}    \right)dB_t \\ 
      &\le -\frac{\pi h_t}{\lambda_t}dt + \sqrt{2} \left( \frac{2 \pi \ld_t - \s_t^2 }{\lambda_t^2}    \right)dB_t\\
 \end{align*}

 \end{proof}
 \par
 In the sequel we will encounter random constants, they will be denoted under the form $c(\omega)$, where $\omega$ stands for the randomness associated to the underlying Brownian motion. This is a generic notation and the exact value of $c(\omega)$ may change from line to line.
 \par
  \begin{proposition}\label{min-rho}
 \CPK{Assume Hypothesis \ref{hypoC02}.}
 Let $(C_t)_{t\in[0,\tau)}$  be the solution of \eqref{eq_sto}, where $\tau$ is any lifetime of \eqref{eq_sto}.    
 Then there exists a positive  random variable $c(\omega)< \infty$ such that for all  $ t < \tau(\omega) $ ,  $ h_t(\omega) \le  c(\omega)$ and 
 $$ \frac{1}{  \frac{1}{\inf \rho_0} +   \sqrt{2} \sup_{[0,t]} B_s + 2 c(\omega)t } \le \inf_{\te} \rho_t(\theta) . $$ 
 
 \end{proposition}
 \begin{proof}
 Let $ J_t = \frac{1}{\rho_t(\theta)} -  \sqrt{2} B_t -2\int_0^t h(s) \,ds - \frac{1}{\inf \rho_0} $. By Lemma \ref{conserve_convex} this quantity is well defined, and by Proposition \ref{prop-ev} we have 
 \begin{align*}
 d J_t (\te) &= -\partial^2_{\theta} \rho_t(\theta)dt - \rho_t(\theta)dt \\
 &= \lt(\rho^2_t(\theta)\partial^2_{\theta} \left( \frac{1}{\rho_t(\theta)} \right) - 2\frac{(\partial_{\theta}\rho_t(\theta))^2 }{\rho_t(\theta)} - \rho_t(\theta)\rt)dt\\ 
 &\le \lt( \frac{1}{ J_t(\te) +  \sqrt{2} B_t  + 2\int_0^t h(s) \,ds + \frac{1}{\inf \rho_0} }   \rt)^2 \partial^2_{\theta} J_t dt.\\  
 \end{align*}
 Using the maximum principle, we will show that $J_t \le 0 $ for all $t \in [0,\tau)$.
 Suppose that  there exists $ t_0 \in [0,\tau) $ and $\te_0 $ such that $\aaaa\df J_{t_0} (\te_0)  > 0 $.
 Let $W_t\df  e^{-t}J_t $, then $W_{t_0}(\te_0) = e^{-t_0} \aaaa >0  $ and $\sup_{\te} W_{t_0}  \ge  e^{-t_0} \aaaa >0  $. Consider the time $t_* = \inf\{t \in [0, t_0], s.t. \quad \sup_{\te} W_t  = W_{t_0}(\te_0) \}   $, and let $\te_*$ such that $W_{t_*}(\te_*) = \sup_{\te} W_{t_*}  $.
 We have  $ t_* >0$ and
 $$  \partial_t W_t \le \lt( \frac{1}{ e^t W_t(\te) +  \sqrt{2} B_t  + 2\int_0^t h(s) \,ds + \frac{1}{\inf \rho_0} }   \rt)^2 \partial^2_{\te} W_t  - W_t   . $$
 Note that since $ 0 \le \partial_t W_t(\te_*)_{\vert t_*}$,  $\partial^2_{\te} W_{t_*}(\te)_{\vert \te_*} \le 0 $ and $W_{t_*}(\te_*) = e^{-t_0} \aaaa >0 $ we get a contradiction.
 Hence for all $t\in [0,\tau)$ we have 
 $$ \frac{1}{\rho_t(\theta)} \le  \frac{1}{\inf \rho_0} +   \sqrt{2} B_t + 2\int_0^t h(s) \,ds . $$ Since $ h_t$ is a positive super martingale by Lemma \ref{prop-h-sm}, it  is almost surely bounded in $ [0,\tau)$, so there exists a positive random variable $c(\omega) < \infty $ such that  $ h_t(\omega) \le  c(\omega)$ and 
 $$ \frac{1}{  \frac{1}{\inf \rho_0} +   \sqrt{2} \sup_{[0,t]} B_s + 2 c(\omega)t } \le \inf_{\te} \rho_t(\theta). $$

 \end{proof}

 \subsection{When there is a sufficient number of symmetries}
 
The goal of this section is to find a necessary condition on the strictly  convex domain to guarantee  the existence of the solution of \eqref{eq_sto} for all times. We will see that the entropy will be a supermartingale if the initial domain has enough symmetries. 
From Lemma \ref{eq_rho_theta}, we deduce  the evolution of the entropy  (defined in \eqref{Ent2}, it also coincides with  the relative entropy of the curvature density with respect to the arc length Lebesgue measure, up to normalizations in terms  of the length of the curve): 
 \bq
  d\Ent_t
&=& d \int_0^{2\pi} \log (\rho_t(\te)) \, d\te\\
& =& -\int_0^{2\pi}  (\partial_{\theta} \rho_t )^2\, d\te dt + 2 \int_0^{2\pi} \rho_t(\theta)   (\rho_t(\theta) - h ) \, d\te dt \\
 &&-  \sqrt{2} \int_0^{2\pi}  \rho_t(\theta) \, d\te  dB_t  .\\
\eq

 \begin{proposition}\label{prop_rho}
 If the boundary of the domain is strictly convex (recall Definition~\ref{strictconv}) then we have the following estimate
    $$ 2 \rho_{\mathrm{inf}} \le h \le 2 \rho_{\mathrm{sup}} $$
 \end{proposition} 
 
 \begin{proof}
 Let $p$ be  the support function, namely  $ p(s) = \langle x(s) , \nu (s) \rangle $.
Green Theorem asserts $ \lambda(D) = \frac12 \int_{\gamma} p(s) ds $
  and we have $\sigma(\partial D) =\int_{\gamma} p(s)\rho(s) ds$. Indeed,
  we compute
  \bq
  \int_{\gamma} p(s)\rho(s) ds&=&
\int_{\gamma} \langle x(s) , \rho(s) \nu (s) \rangle  ds\\ & =& -\int_{\gamma} \langle x(s) , x(s)'' \rangle  ds   \eq
 and it remains to integrate  by part to recognize $\sigma(\partial D)$.\par
  Remark also that we can suppose that the origin is contained in the domain (else translate and all the quantities are invariant under translation). By convexity of the domain we have that $p(\te) > 0$.
Recalling that $ d\te = \rho ds $, we have $\sigma(\partial D) =\int_{\gamma} p(\te)  d\te $, so that
   $$ \begin{aligned}
            \lambda(D) & = \frac12 \int_{\mathbb{T}}  \frac{p(\te)}{\rho(\te)}  d\te  \le \frac{1}{ 2 \rho_{\mathrm{inf}}} \sigma(\partial D)\\ 
   \end{aligned} .$$ 
   Hence $ 2 \rho_{\mathrm{inf}} \le h(D)$. The other inequality is a direct consequence of Gage inequality.
 \end{proof}
 \begin{proposition}\label{prop-wirt}
 For any $\mathcal{C}^1$ function $f\st [0,L]\ri\RR$ satisfying
  $f(0)= f(L)= 0$, we have
 $$\int_{0}^{L}  f^2 d\te \le \lt(\frac{L}{\pi}\rt)^2  \int_{0}^{L}  f'^2 d\te  $$
 \end{proposition}
 \begin{proof}
 This is the Wirtinger inequality which can be proved by  Fourier series. 
 \end{proof}
 
\begin{definition}
We will say that a domain $D$ has  $ n $ axes of symmetries, if  up  to a translation there exists a linear straight line $\Delta$ such that $D$ is symmetric with respect to $\Delta,R_{\pi/n}(\Delta)...,R_{(n-1)\pi/n}(\Delta) $, where $ R_{\te}$ is a rotation of angle $\te$.
\end{definition}

 \begin{proposition}\label{Ent}
\CPK{Under Hypothesis \ref{hypoC02}  and the assumption that $D_0$}  has  $ n $ axes of symmetries, with $n\ge 3$, the entropy is a  
 super-martingale. 
 \end{proposition}
 \begin{proof}
 Using proposition \ref{prop_rho}  and the symmetries there exists $\te_k \in [\frac{k\pi}{n},\frac{(k+1)\pi}{n}  )$  for $ k \in\lin 0, 2n-1\rin$ such that $\rho(\te_k) = \frac{h}{2} $. Note that we can further impose that $ \vert  \te_k - \te_{k-1} \vert \le \frac{2\pi}{n} $ for $k \in\lin 0, 2n\rin $ (with $\te_{2n}=\te_0 + 2\pi$).
 
 So using Proposition \ref{prop-wirt} ,  we get
 $$\int_{ \te_k }^{ \te_{k+1} }  \lt(\rho(\te)- \frac{h}{2}\rt)^2 d\te \le \frac{4 }{n^2} \int_{ \te_k }^{ \te_{k+1} }  \rho'(\te)^2 d\te .  $$
 Hence
 $$ -\int_{\mathbb{T}}  \rho'(\te)^2  d\te  \le -\frac{n^2}{4} \int_{\mathbb{T}}  \lt(\rho(\te)- \frac{h}{2}\rt)^2  d\te  ,$$
 and, if $ n \ge 3$ we have 
  $$ \begin{aligned}
  d\Ent_t  \le& -\frac{n^2}{4} \int_{\mathbb{T}}  \lt(\rho(\te)- \frac{h}{2}\rt)^2  d\te  + 2 \int_0^{2\pi}   \lt(\rho_t(\theta) - \frac{h}{2} \rt)^2 \, d\te dt - \pi h^2 dt  \\
 &-  \sqrt{2} \int_0^{2\pi}  \rho_t(\theta) \, d\te  dB_t  .\\
 \le &  - \pi h^2 dt   -  \sqrt{2} \int_0^{2\pi}  \rho_t(\theta) \, d\te  dB_t  \\
 \end{aligned} $$
 \end{proof}
 
\begin{Remark}
The Green-Osher's inequality, see Theorem 0.2 of \cite{MR1793675},   shows
\bq \Ent_t &= & \int_{\mathbb{T}} \ln ( \rho_t) d\te\ \ge\ \pi \ln \lt( \frac{\pi}{ \ld_t  }  \rt) .\eq
Since $ \frac{1}{\ld_t} $ is a positive martingale, the r.h.s.\ is a super-martingale (at least on its domain of definition).
Of course, this is not sufficient to insure that $(\Ent_t)_t$ itself is a super-martingale.
\end{Remark}
 In the sequel we will use comparison of processes up to a continuous martingale term: when $(X_t)_{t\in[0,\tau)}$ and $(Y_t)_{t\in[0,\tau)}$ are two predictable processes with respect to the same underlying filtration and are defined on the same  time-interval $[0,\tau)$ (where $\tau$ is a stopping time), we write 
  $$ \fo t\in[0,\tau)\qquad X_t\leqm Y_t $$
 to mean there exists a continuous martingale $(M_t)_{t\in[0,\tau)}$ such that
 
 $$ \fo t\in[0,\tau)\qquad X_t \leq Y_t+M_t $$
 
 In the next four results, $\tau$ will stand the maximal time up to which the equation of Lemma \ref{eq_rho_theta} admits a solution.
 \begin{proposition} \label{p_4.10}  \CPK{Under Hypothesis \ref{hypoC02}, we have}:
\bq d \int (\partial_{\te}\rho_t)^2  d\te &\leqm&  \lt(\lt(\frac{13}{3}\rt)^2+16\rt) \int \rho^4 \,d\te dt\eq
 \end{proposition}
 \begin{proof}
 From Lemma \ref{eq_rho_theta}, we deduce that on $[0,\tau)$, via integrations by parts,
 \bq
\lefteqn{ d\int (\pa \rho)^2}\\&=&2\int \pa \rho \,d\pa\rho+\int d\pa\rho\, d\pa\rho \\
 &=&2\int \pa \rho \, \pa d\rho+\int \pa d\rho\, \pa d\rho \\
 &=&-2\int \pa^2\rho \, d\rho +2\int (\pa \rho^2)^2 dt\\
 &=&2\int \rho^2\pa^2\rho [(2h-3\rho-\pa^2\rho)dt+\sqrt{2}dB_t]+8\lt(\int (\rho^2\pa\rho)\pa\rho \rt) dt\\
 &=&2\lt(\int \rho^2\pa^2\rho (2h-3\rho-\pa^2\rho)\rt)dt-\f83\lt(\int \rho^3\pa^2\rho \rt) dt\\
 &&+2 \sqrt{2}\lt(\int \rho^2\pa^2\rho \rt)dB_t\\
 &\stackrel{(m)}{=}& 2\lt(\int\rho^2\lt[-(\pa^2\rho)^2+\lt(2h-\f{13}{3}\rho\rt)\pa^2\rho\rt]\rt) dt\\
 &=& 2\lt(\int\rho^2\lt[\lt(h-\f{13}{6}\rho\rt)^2-\lt(\pa^2\rho+\f{13}{6}\rho-h\rt)^2\rt]\rt) dt\\
 &\leq &2\lt(\int\rho^2\lt(h-\f{13}{6}\rho\rt)^2\rt) dt\\
 &\leq &4\lt(\int \lt(\f{13}{6}\rt)^2\rho^4+h^2\rho^2\rt) dt\\
 \eq
 where $\pa$ stands for the differentiation with respect to the underlying parameter $\theta$ (which commutes with respect to the ``stochastic differentiation with respect to time'' $d$).\par
 Taking into account Gage's inequality, we get
 \bq
 h^2\int \rho^2&\leq & \f1{\pi^2}\lt(\int\rho\rt)^2\int \rho^2\\
 &\leq &4 \int\rho^4\eq
 and finally the desired bound.
  \end{proof}
 \par
This observation leads us to investigate the evolution of $\int\rho^4$ itself:
   \begin{proposition} \CPK{Under Hypothesis \ref{hypoC02}  and the assumption that $D_0$  has $n$ axes of symmetries, with $n \ge 7 $, we have}
\bq d \int (\rho_t)^4  d\te &\leqm&  c(\omega) \,dt\eq
where $c(\omega)$ is a finite random constant (independent of time), as mentioned before Proposition \ref{min-rho}.
 \end{proposition}
\begin{proof}
We compute
\bq
d\int\rho^4&=&4\int \rho^3d\rho+6\int \rho^2d\rho\, d\rho\\
&\stackrel{(m)}{=}&4\lt(\int 6\rho^6+\rho^5\pa^2\rho-2h\rho^5\rt) dt\\
&=&4\lt(\int 6\rho^6-5\rho^4(\pa\rho)^2-2h\rho^5\rt) dt\\
&=&4\lt(\int 6\rho^6-\f59\lt(\pa\rho^3\rt)^2-2h\rho^5\rt) dt
\eq\par
To deal with the middle term, let us resort to Wirtinger inequality, assuming $n\geq 7$ axes of symmetry for $D_0$.
Since the evolution equation is  invariant by these symmetries, for any time $t\in[0,\tau)$, we still have that $D_t$ has $n$ axes of symmetry.
We deduce that
\bq
\int\lt(\pa\rho^3\rt)^2&=&\int\lt(\pa(\rho^3-\rinf^3)\rt)^2\\
&\geq & \f{49}{4}\int\lt(\rho^3-\rinf^3\rt)^2
\eq
so that, taking into account Proposition \ref{prop_rho},
\bq
d\int\rho^4&\leqm&4\lt(\int -\f{29}{36}\rho^6+\f{245}{18}\rho^3\rinf^3-\f{245}{36} \rinf^6-2\rho^5h\rt)dt
\\&\leq & 2\lt(\int -\f{29}{18}\rho^6+\f{245}{9}\rho^3\rinf^3-\f{389}{18}\rinf^6\rt)dt\\
&\leq & c(\omega)  \,dt
\eq
\par
To get the desired result, recall that $\rinf\leq h/2$ and that $h$ is a positive supermartingale and is thus a.s.\ bounded on its domain of definition.
\end{proof}
\par
 \begin{proposition}\label{prop_rho_prime}
  \CPK{Under  Hypothesis \ref{hypoC02} and the assumption that} has $n$ axes of symmetries  with $n \ge 7 $,  there exists a finite random variable $ c(\omega)$ such that on  the event $\tau < \infty $:
\bqn{fini} \forall t \in [0, \tau),\qquad \int (\partial_{\te}\rho_t)^2 d\te & \le &c(\omega)
\eqn
 \end{proposition}
 \begin{proof}
 Let us show that there exists a finite random variable $ c_1(\omega)$ such that on  the event $\tau < \infty $:
 \bqn{fini1} \forall t \in [0, \tau),\qquad  \int (\rho_t)^4  d\te& \le &c_1(\omega)
 \eqn
According to the  previous proposition, there exist a finite random constant $c(\omega)\geq 0$ and a continuous martingale $(M_t)_{t\in[0,\tau)}$ such that

\bq \fo t\in[0,\tau),\qquad   \int (\rho_t)^4  d\te &\leq & c(\omega) t+M_t\eq
\par
Up to enriching the underlying probability space, we can find a Brownian motion $(W_t)_{t\geq 0}$ such that
\bqn{majo}
\fo t\in[0,\tau),\qquad \int (\rho_t)^4  d\te&\leq & c(\omega) t+W_{\lan M\ran_t}\eqn\par
Thus on $\{\tau<+\iy\}$, we will deduce \eqref{fini1} as soon as we show
\bq
\lim_{t\ri\tau-}\lan M\ran_t&<&+\iy\eq
\par
Note that if we had 
\bq
\lim_{t\ri\tau-}\lan M\ran_t&=&+\iy\eq
we would get from \eqref{majo} that
\bq
\inf_{t\in[0,\tau)}   \int (\rho_t)^4  d\te  &=&-\iy\eq
which is a contradiction. Hence there exists $ c_1(\omega)$ such that \eqref{fini1} is satisfied  on  the event $\tau < \infty $.
According to  Proposition \ref{p_4.10} there exist a finite constant $c_2(\omega)\geq 0$ and a continuous martingale $(\tilde{M}_t)_{t\in[0,\tau)}$ such that on $ \{\tau<+\iy\}$
\bq \fo t\in[0,\tau),\qquad   \int (\partial_{\te}\rho_t)^2 d\te &\leq & c_2(\omega) t+\tilde{M}_t\eq
\par
We deduce \eqref{fini} by the same argument used to get \eqref{fini1}.
\end{proof}
\par
 \begin{proposition} \label{prop_rho_bound}
   \CPK{Under Hypothesis \ref{hypoC02} and the assumption that $D_0$} has $n$ axes of symmetries,  with $n \ge 7 $, there exists a  random variable $ c(\omega)$ such that on  the event $\tau < \infty $:
   $$ \rho_t \le c(\omega) < \infty \quad \forall t \in [0,\tau). $$
 \end{proposition}
 \begin{proof}On the event  $\tau < \infty $,
 according to Propositions \ref{prop_rho_prime} and \ref{Ent}, there exists a random constant $c(\omega) < \infty $  such that for all $ t < \tau $ we have:
 $$ \Ent_t \le c(\omega)  $$
 $$ \int (\partial_{\te}\rho_t)^2 \le c(\omega) .$$
 Let $r_t\df  \sup \{ \rho_s(\te), (\te,s) \in [0,2 \pi] \times [0, t]\}$ for $t < \tau. $
 Then there exists $(\te_1,t_1) \in [0,2 \pi] \times [0, t] $ such that $ \rho_{t_1}(\te_1) = r_t $.
 For all $ \te_2 \in [0,2 \pi] $, we have 
 \bq
  \vert \rho_{t_1}(\te_1) - \rho_{t_1}(\te_2 ) \vert &= &\lt\vert \int_{\te_1}^{\te_2} \partial \rho_{t_1} (\te) \, d \te  \rt\vert  \\
   &\le &\sqrt{\vert \te_1 - \te_2  \vert } \sqrt{c(\omega)},
 \eq
 so  $$ r_t - \sqrt{\vert \te_1 - \te_2  \vert } \sqrt{c(\omega)} \le  \rho_{t_1}(\te_2 ). $$
 Then using Proposition \ref{min-rho} we get
 \bq
 \Ent_{t_1} &\ge& \int_{ \vert \te - \te_1 \vert \le \frac{r_t^2 }{4 c(\omega)} \wedge \frac{\pi}{2}} \log(\rho_{t_1} (\te)) \, d \te +  \int_{ \vert \te - \te_1 \vert \ge \frac{r_t^2 }{4 c(\omega)} \wedge \frac{\pi}{2}} \log(\rho_{t_1} (\te)) \, d \te \\
 & \ge& 2\log\lt(\frac{r_t}{2}\rt) \lt(\frac{r_t^2 }{4 c(\omega)} \wedge \frac{\pi}{2}\rt)\\&& + 
 \lt( 2\pi - 2\lt(\frac{r_t^2 }{4 c(\omega)} \wedge \frac{\pi}{2}\rt)\rt)  \log  \lt( \frac{1}{  \frac{1}{\inf \rho_0} +   \sqrt{2} \sup_{[0,t_1]} B_s + 2 c(\omega)t_1 }  \rt)\\ 
 & \ge& 2\log\lt(\frac{r_t}{2}\rt) \lt(\frac{r_t^2 }{4 c(\omega)} \wedge \frac{\pi}{2}\rt)\\&& + 
 \lt( 2\pi - 2\lt(\frac{r_t^2 }{4 c(\omega)} \wedge \frac{\pi}{2}\rt)\rt)  \log  \lt( \frac{1}{  \frac{1}{\inf \rho_0} +   \sqrt{2} \sup_{[0,\tau]} B_s + 2 c(\omega)\tau }  \rt).\\ 
 \eq
 On the event $\tau < \infty $ the last term of the above equation is  almost surely bounded, since the entropy is bounded from above on $[0,\tau) $. { We get that $ \rho_t$ has to be a.s. uniformly bounded  on $t \in[0, \tau). $}
 
 \end{proof}

We will need the following lemma which is a little refinement of Lemma 4.1.1 from \cite{Gage_Hamilton}.
 
 \begin{lemma}\label{ident}
Let a $2 \pi $ periodic positive function $\rho \in C^{\alpha }(\mathbb {T}) $, with $\alpha\in(0,1)$, satisfying \eqref{cores}.
Consider the curve $ X : \theta \mapsto (\int_0 ^{\theta} \frac{\cos(u)}{\rho(u)} \,du , \int_0 ^{\theta} \frac{\sin(u)}{\rho(u)} \,du )  $, as before parametrized by the angle $\theta\in\TT$ of its tangent
with respect to the horizontal axis, and whose curvature function is $\rho$. When $X$ is parametrized by its arc-length, it becomes
 $C^{2+\alpha}$.
 \end{lemma}
 \begin{proof}
Under the parametrization of  $X $ by $\theta $, the curve may seem to be only of order $C^{1 + \alpha}. $ Let us check it is in fact $ C^{2+\alpha}$ under the arc-length parametrization. 
 Denoting $ s  $ the arc length parametrization of $X$, we have
  $\partial_s = \rho \partial_{\te} $ and  $ s(\te) =  \int_0^{\te }\frac{1}{\rho(u)}\, du$, $\partial_s\te(s) = \rho(\te(s)) $, $T(s) = (\cos(\te(s)), \sin(\te(s)) $ (as it should be, by definition of the parametrization by $\theta$). 
  From $\partial_s\te(s) = \rho(\te(s)) $, we see that $s\mapsto \te(s)$ is $C^{1+\alpha}$.
Furthermore, in the parameter  $s$, the curve $ \tilde{X}(s) \df X(\te(s))$ satisfies $ \partial_s \tilde{X} =(\cos(\te(s)),\sin(\te(s)) $,   so we get that $\tilde{X}$ is $ C^{2+\alpha}$.
  \end{proof}

 \begin{thm}\label{infinite-lifetime}
\CPK{Under Hypothesis \ref{hypoC02} and the assumption that $D_0$}   has $n$ axes of symmetries,  with $n \ge 7 $, a.s. $\tau = \infty $, where $\tau  $ is the maximal  lifetime of \eqref{eq_sto}. 
 \end{thm}
 \begin{proof}
 Suppose that $ \mathbb{P}(\tau < \infty)  > 0 . $
 Let $C_t(\te)$ be the solution of \eqref{eq_sto} namely 
  \begin{equation*} \left \{  \begin{array}{lcl}
 d_t C(t,\te) & = & \lt([-\rho_t(C(t,\te)) + 2h_t ]dt + \sqrt{2}dB_t \rt) \nu_t(C(t,\te))  \\
  C(0,\te) &= & C_0(\te) \\
  \end{array}
  \right. 
  \end{equation*}
  On the event $ \{ \tau < \infty \}$, using Lemma \ref{prop-h-sm} and \ref{prop_rho_bound} we have for all $ t < \tau $, $h_t \le c(\omega) < \infty $ and $ \rho_t(\te) \le c(\omega)< \infty $.  Since $\Vert \nu_t(C(t,\te)) \Vert = 1 $  we have for  
  $ s, t < \tau $ such that $ \vert t-s \vert $ is small:
  $$ \vert C(s,\te) - C(t,\te) \vert  \le c_1(\omega) \vert t-s \vert ^{\frac12 - \epsilon}, $$
  where the random variable $c_1 $ depends on $ c$.
     Hence there exists $ C_{\tau} : \mathbb{T} \mapsto \mathbb{R}^2$ such that $ C_t  $  converges uniformly  to $C_{\tau} $.
  On the other hand, using Proposition \ref{prop_rho_prime} we get by H\"older inequality that for all $ t < \tau $
  $$  \vert \rho_t(\te) - \rho_t(\beta) \vert \le \left\vert \int_{\beta}^{\te} \partial \rho_t(\gamma) d\gamma  \right\vert \le c(\omega)\sqrt{\vert \te - \beta \vert}.$$
  Hence $ \rho_{.} $ is equi-continuous. So using again Proposition  \ref{prop_rho_bound} and Ascoli Theorem we get that there exists a sequence $(t_n)_{n}$ converging to  $\tau $ and a $ C^{\frac12}$ function $ \rho_{\tau}$ such that $\rho_{t_n} $ converges uniformly to $ \rho_{\tau}$.
  
  We want to show that $C_{\tau}$, is in fact $C^{2+\frac12}$.
   
 By Theorem \ref{th_cores} we have the following representation of the solution  of \eqref{eq_sto}:
 $${C}(t,\te)\df  \tilde{C}(t,\te) + \int _0^t (-\partial_{\te}\rho_u(0), \rho_u(0)-2h_u ) \,du - (0, \sqrt{2} B_t  ) $$
where
$$ \tilde{C}(t,\theta) =  \left( \int_0^{\theta} \frac{\cos(\theta_1)}{\rho_t(\theta_1)} \, d \theta_1 ,  \int_0^{\theta} \frac{\sin(\theta_1)}{\rho_t(\theta_1)} \, d \theta_1   \right). $$
Since $C(t_n, 0) \to C_{\tau}(0)$,  there exists $A \in \mathbb{R}^2$ such that  
$$  \int _0^{t_n} (-\partial_{\te}\rho_u(0), \rho_u(0)-2h_x ) \,du - (0, \sqrt{2} B_{t_n}  ) \to A .$$ 
 Also since $\rho_{t_n} $ converges uniformly to $ \rho_{\tau}$ and by Proposition \ref{min-rho}, $ \rho_{\tau}>0$, we have that $\tilde{C}(t_n, .) $ converges uniformly to 
 $ \left(\int_0^{.} \frac{\cos(\theta_1)}{\rho_{\tau}(\theta_1)} \, d \theta_1 ,  \int_0^{.} \frac{\sin(\theta_1)}{\rho_{\tau}(\theta_1)} \, d \theta_1   \right) .$
 Hence $$ C(t_n, .) \to  \left(\int_0^{.} \frac{\cos(\theta_1)}{\rho_{\tau}(\theta_1)} \, d \theta_1 ,  \int_0^{.} \frac{\sin(\theta_1)}{\rho_{\tau}(\theta_1)} \, d \theta_1   \right) + A = C_{\tau}(.)  $$  
 
 By Lemma \ref{ident} we get that the curve $C_\tau $ is $C^{2+\frac12}$. Using Theorem 22 in \cite{zbMATH07470497}, and the  Markov property we can extend the solution after the time $\tau$ by a solution starting at the curve $C_\tau $, which is in contradiction with the maximality of $\tau$.
   
 \end{proof}
 
 We have the following corollary of Theorem 61 in \cite{zbMATH07470497}.
 \begin{corollary}\CPK{Consider $(D_t)_{t\geq 0}$ the solution of \eqref{eq_sto}. Under Hypothesis \ref{hypoC02} and the assumption that $D_0$}  has $n$ axes of symmetries,  with $n \ge 7 $,  
we have a.s.\ in the Hausdorff metric,
\bq
\lim_{t\rightarrow +\infty} \frac{D_t}{\sqrt{\lambda(D_t)}}&=&B(0,1/\sqrt{\pi})\eq
where $B(0,1/\sqrt{\pi})$ is the Euclidean  ball centered at 0 of radius $1/\sqrt{\pi}$.
 \end{corollary}
 
 \section{Symmetric convex sets in $\mathbb{R}^2$ with star shaped skeletons}
\label{conv-sym}
Let $\mathcal{C}$ be the set of smooth closed simple and strictly convex curves embedded in $\mathbb{R}^2$. 

Fix $n\ge 2$.
Let $\mathcal{T}_n$ be the set of closed curves  symmetric with respect to the vertical axis, denoted  $\Delta$, and invariant by the rotation $R_{2\pi/n}$ of angle $2\pi/n$ (and thus invariant by the group $G_n$ generated by these two isometries).\par
  
  Let us describe the set  $\mathcal{C}$ in terms of its curvature. Let $C_0 \in \mathcal{C} $, and let $ C : \TT \to \mathbb{R}^2$ be the parametrization of $C_0$ such that  $\theta$ is the angle between the tangent line and the $x$ axis at the point $C(\theta)$ i.e a tangent vector is $ (\cos(\theta),\sin(\theta))\in T_{C(\theta)}C$. Note that this parametrization is possible since $ \partial_s \theta = \rho(\theta) >0$ where $s$ is the arc-length parametrization (due to Fr\'enet equation). From now on, we will take this parametrization for   curves in $\mathcal{C}$.

Recall from Lemma 4.1.1 of  Gage and Hamilton \cite{Gage_Hamilton}
  that  a $2 \pi $ periodic positive function $\rho $ represents the curvature of a simple closed strictly convex plane curve if and only if 
 $I_{c,\rho}(2\pi) = I_{s,\rho}(2\pi) = 0$, where

$$ 
 I_{c, \rho}(\beta) \df \int_0 ^{\beta} \frac{\cos(\te)}{\rho(\te)} \,d\te , \quad
 I_{s, \rho}(\beta) \df  \int_0 ^{\beta} \frac{\sin(\te)}  {\rho(\te)} \,d\te \quad
  \beta  \in \TT .
 $$
 
More precisely, we have
\bq
 \mathcal{C}  &\simeq & \{ \rho \in \rC^1(\TT, ]0,\infty))\st I_{c,\rho}(2\pi) = I_{s,\rho}(2\pi) = 0 \} \times \mathbb{R}^2 \eq
 through the reciprocal bijections given by
 \bq
 \{ \theta \mapsto C(\theta)  \} &\longmapsto& (\{ \theta \mapsto \rho(\theta) \}, C(0))\\
  \{\theta \mapsto (I_{c, \rho}(\theta) , I_{s, \rho}(\theta) ) + X \} &\longmapsfrom& (\{ \theta \mapsto \rho(\theta) \}, X) \\
\eq
 
 Let us describe the set  $\mathcal{C} \cap \mathcal{T}_n  $ in terms of its curvature. Let $C \in \mathcal{C} \cap \mathcal{T}_n  $. For any $\theta\in\TT$, denote $S_\theta$ the symmetry with respect to $R_\theta(\Delta)$.
 Using the symmetry $S_0$ we have $C(-\theta)=S_0(C(\theta))$ implying that $C(0)=S_0(C(0))$ and
\begin{equation}\label{b}
 C(0) = (0, -b)\quad\hbox{for some}\quad  b\ge 0.
\end{equation}
  Using the symmetry   $S_{\pi/n}=R_{2\pi/n}S_0$ (thus belonging to $G_n$) we have:
 $C(2\pi/n - \theta) = S_{\pi/n}(C(\theta))  $,  yielding for $\theta=\pi/n$:  
\begin{equation}\label{a}
C\left(\frac{\pi}{n}\right) = R_{\pi/n}((0,-a))\quad\hbox{ for some}\quad  a\ge 0.
\end{equation}
 The two numbers $b,a$ are positive  since $ (0,0) \in \mathrm{int} (C)$ by convexity. 
 Also $C$ is completely defined by its restriction to $[0,\frac{\pi}{n}] $. Using the invariance by $G_n$  we have the following property of the associated curvature 
 \begin{equation}\label{sym}
  \rho \left(\theta +  \frac{\pi}{n}\right) = \rho\left(\frac{\pi}{n} - \theta \right),  \theta \in \lt[0, \frac{\pi}{n}\rt], \quad\hbox{ and }\rho\hbox{ is }\frac{2\pi}{n}\hbox{-periodic.}
 \end{equation}

   So $ \partial_{\theta}\rho\left(\frac{k\pi}{n}\right)  = 0 $ for all $k\in\{0,\ldots, 2n-1\}$.

 A fundamental object for the study of elements of $\mathcal{C} \cap \mathcal{T}_n$ will be the projection to some well chosen lines. Let $C\in \mathcal{C} \cap \mathcal{T}_n$. For $\theta\in (  0, \frac{\pi}{n}]$, 
 let $(0,\Pi(\theta) )$ be the intersection of the line $ D_{\te}$  orthogonal to $C$ at the point $ C(\theta)$ and the vertical axis~$\Delta$. Define 

\begin{equation}\label{Sn}
\mathcal{S}_n := \left\{ C\in \mathcal{C} \cap \mathcal{T}_n, \ \Pi \ \hbox{is \strictly increasing on}\ \left[0,\frac{\pi}{n}\right]\right\}.
\end{equation}

Define also
\begin{equation}\label{Snrho}
\mathcal{S}_n^\downarrow := \left\{ C\in \mathcal{C} \cap \mathcal{T}_n, \ \rho \ \hbox{is \strictly decreasing on}\  \left[0,\frac{\pi}{n}\right]\right\}.
\end{equation}
Notice that for $C\in \mathcal{S}_n $ or $C\in \mathcal{S}_n^\downarrow $,
 since $C \in \mathcal{C}\cap  \mathcal{T}_n  $, it is characterized by its values for $\theta \in [  0, \frac{\pi}{n}] $. 
 \begin{proposition}\label{proj-skeleton}
 	Let $C\in \mathcal{S}_n $.
  Then  $\Pi(\pi/n)=0$, and $\Pi$ has a limit $-y_0<0$ as $\theta\searrow0$, so it extends to a $C^1$ nonpositive non-increasing function on $[0,\pi/n]$. 
 \end{proposition}
\begin{proof}
	Since the outward normal at $C(\theta)$ is $ \nu(\te)\df  (\sin(\theta), -\cos(\theta)) $ we have for all $\theta \in (  0, \frac{\pi}{n}) $
	$$ \Pi(\theta) = -b+  \int_0 ^{\te} \frac{\sin(\beta)}{\rho(\beta)} \,d\beta + \cot(\te)\int_{0}^{\theta} \frac{\cos(\beta)}{\rho(\beta)} \,d\beta= -b +\int_0^\theta \frac{\cos(\theta-\beta)}{\rho(\beta)\sin(\theta)} \, d\beta$$
with $b$ defined in~\eqref{b}, so 
	\begin{equation}\label{limPi}
	\lim_{ \te \searrow 0} \Pi(\theta) = -b + \frac{1}{\rho(0)}=:-y_0.
	\end{equation}
On the other hand, by symmetry of $C$, the point $(0,\Pi(\pi/n))$ also belongs to $R_{2\pi/n}(\Delta)$, so $\Pi(\pi/n)=0$. As a consequence, since we have assumed that $\Pi$ is non-decreasing, we have $y_0>0$ and $\Pi$ is negative on $[0,\pi/n)$.

From now on we let $\Pi(0):=-y_0$.

	 Using an integration by part we have for $\theta\in (0,\pi/n)$
  \bqn{pi-prime} 
 \nonumber \Pi'(\te) &= & \frac{1}{\rho(\te)\sin(\te)} - \frac{1}{\sin^2(\te)} \int_{0}^{\te} \frac{\cos(\beta)}{\rho(\beta)} \,d\beta      \\
  \nonumber  &=& \frac{1}{\rho(\te)\sin(\te)} + \frac{1}{\sin^2(\te)}  \lt( \lt[-\frac{\sin(\beta)}{\rho(\beta)}\rt]^{\te}_{0} - \int_{0}^{\te} \frac{\rho'(\beta) \sin(\beta)}{\rho^2(\beta)} \, d\beta  \rt) \\
            &=& \frac{-1}{\sin^2(\te)} \int_{0}^{\te} \frac{\rho'(\beta) \sin(\beta)}{\rho^2(\beta)} \, d\beta.  
  \eqn
 Note that $$\Pi \in \rC^1\left(\left(0,\frac{\pi}{n}\right]\right) \cap   C^0\left(\left[0,\frac{\pi}{n}\right]\right)  .$$ Taking into account that $ \lim_{\te \searrow 0}\Pi'(\te) = 0 $, due to $ \rho'(0) = 0$, we end up with $\Pi \in \rC^1([0,\frac{\pi}{n}])$. 
\end{proof}

The following result is a direct consequence of Equation~\eqref{pi-prime}:
\begin{proposition}\label{rhotoPi}
We have $\mathcal{S}_n^\downarrow\subset \mathcal{S}_n$.
\end{proposition}

Consider the mapping $r$ defined by
\bq\fo \theta\in[0,\pi/n],\qquad r (\theta ) &\df& \Vert C(\te) - (0,\Pi(\te)) \Vert \eq \par Since the curve does not cross the vertical axis before $\frac{\pi}{n} $,  $r \in \rC^1((0,\frac{\pi}{n}])\cap C^0([0,\frac{\pi}{n}])  $. Hence we have the following parametrization of the curve $C$, for $\theta \in (  0, \frac{\pi}{n}] $
\bqn{Ctheta} C(\te) &=& (0,\Pi(\te)) + r(\te)(\sin(\te),-\cos(\te))\eqn
\begin{lemma}\label{Pi-r-param}
The map $\theta\mapsto (\Pi(\theta), r(\theta))$  extends to a $C^1$ map defined on $[0,\pi/n]$, and satisfying $$(\Pi(0), r(0))=\left(-b+\f1{\rho(0)},\f1{\rho(0)}\right).$$
\end{lemma}
\begin{proof}
We are only left to prove the assertion for the map $r$.
 Putting the two parametrizations together, since $\langle C'(\te), N(\te) \rangle = 0$ and   $\langle C'(\te), T(\te) \rangle = \frac{1}{\rho(\te)}$, from \eqref{vrho}, we deduce from \eqref{Ctheta} that for any $\theta\in(0,\pi/n]$,
 $$\lim_{ \te \searrow 0}  r(\te) = \frac{1}{\rho(0)},$$ 
 \begin{equation}\label{r} \left\{ 
  \begin{array}{lcr}
  -\Pi'(\te)\cos(\te) + r'(\te) &=& 0 \\
 \Pi'(\te)\sin(\te) + r(\te) &= & \frac{1}{\rho(\te)}, \\
  \end{array} 
 \right.  
 \end{equation}
i.e
\begin{equation} \left\{
  \begin{array}{lcr}
 \begin{pmatrix}
\Pi(\te)\\
r(\te)
\end{pmatrix}^{'}
+
\begin{pmatrix}
 0 & \frac{1}{\sin(\te)}\\
 1 & \cot{\te}
\end{pmatrix}
 \begin{pmatrix}
0\\
r(\te)
\end{pmatrix}
=
\begin{pmatrix}
 0 & \frac{1}{\sin(\te)}\\
 1 &\cot{\te}
\end{pmatrix}
 \begin{pmatrix}
0\\
\frac{1}{\rho(\te)}
\end{pmatrix} \\
\begin{pmatrix}
\Pi\left(\frac{\pi}n\right)\\
r\left(\frac{\pi}n\right)
\end{pmatrix}
=
\begin{pmatrix}
0\\
a
\end{pmatrix}
\end{array} 
 \right.   
 \end{equation}
where $a$ is defined in~\eqref{a}.
Using the first equation of \eqref{r}, we get that $\lim_{\te \searrow 0} r'(\te) =0,$ so $r \in \rC^1([0,\frac{\pi}{n}]) $.
\end{proof}

\begin{proposition} \label{caract}
Let $C$ be a curve in  $\mathcal{C} \cap \mathcal{T}_n$. 
\begin{enumerate}
\item If   $\Pi$ is non-decreasing on $[0,\pi/n]$ (i.e. if $C\in \mathcal{S}_n$), then the skeleton of $C$ is $G_n(\{0\}\times [-y_0,0]).$
\item If  the skeleton of $C$ is  $G_n(\{0\}\times [-y,0])$ 
then $\Pi$ is non-decreasing.
\end{enumerate}
\end{proposition} 
\begin{proof}
(1) First assume that $C\in \mathcal{S}_n$.
Denote by $S$ the skeleton of $C$. 

a) First we prove that $G_n(\{0\}\times [-y_0,0])\subset S$. For this it is sufficient to prove that for all $\theta\in [0,\pi/n]$, the point $(0,\Pi(\theta))$ belongs to $S$. 

 We only need to prove it for $\theta\in (0,\pi/n)$ since the skeleton is closed. For the same reason we can also assume that $\Pi'(\theta)>0$. So let $\theta\in (0,\pi/n)$ with $\Pi'(\theta)>0$. The closed disk $\bar B((0,\Pi(\theta)), r(\theta))$ centered at $(0,\Pi(\theta))$ and with radius $r(\theta)$ meets $C$ at least at the two points $C(\theta)$ and $C(-\theta)$. To prove that $(0,\Pi(\theta))\in S$ we need to prove that it is inside $\bar D$. This will be done in two steps.
\begin{itemize}
\item
 We prove that the set 
$$
\left\{ (0,\Pi(\theta))+r (\cos\varphi, \sin\varphi) |\ 0\le r\le r(\theta), \ -\frac\pi2-\frac{\pi}n\le \varphi\le -\frac\pi2+\frac{\pi}n\right\}
$$
is included in $\bar D$.\\
 The proof is by contradiction,
assume there exists $\theta'\in [0,\theta)$ such that  
$
\Vert C(\te') - (0,\Pi(\te)) \Vert =r(\theta)
$. Consider the closed disk $\cC$ centred at $(0,\Pi(\theta))$ of radius $r(\theta)$, passing through $C(\theta)$ and $C(\theta')$, see Figure \ref{cC1}.

On one hand, by~\eqref{r} we have  $r(\theta)< \frac1{\rho(\theta)}$, so for $\alpha<\theta$ and $\alpha$ close to $\theta$, the points $C(\alpha)$ are outside
the disk $\cC$. On the other hand, since $\Pi(\alpha)<\Pi(\theta)$,
there exists $ q_{\alpha} \in D_{\alpha} \cap ((0,\Pi(\te)),C(\te')]$. As we can see in the proof of the point b) below,  $C(\alpha) $
is the nearest point of $q_{\alpha}$ in $C$. We have $ \Vert q_{\alpha} - C(\alpha) \Vert \le  \Vert q_{\alpha} - C(\te') \Vert$. Hence 
\begin{align*}
 \Vert C(\alpha) - (0, \Pi(\te) ) \Vert &< \Vert q_{\alpha} - C(\alpha) \Vert + \Vert q_{\alpha} - (0, \Pi(\te) ) \Vert \\
 & \le \Vert q_{\alpha} - C(\te') \Vert + \Vert q_{\alpha} - (0, \Pi(\te) ) \Vert = r(\te) \\ 
\end{align*}
and we get a contradiction.

\begin{figure}[ht]
\caption{\label{cC1}}
\begingroup%
  \makeatletter%
  \providecommand\color[2][]{%
    \errmessage{(Inkscape) Color is used for the text in Inkscape, but the package 'color.sty' is not loaded}%
    \renewcommand\color[2][]{}%
  }%
  \providecommand\transparent[1]{%
    \errmessage{(Inkscape) Transparency is used (non-zero) for the text in Inkscape, but the package 'transparent.sty' is not loaded}%
    \renewcommand\transparent[1]{}%
  }%
  \providecommand\rotatebox[2]{#2}%
  \newcommand*\fsize{\dimexpr\f@size pt\relax}%
  \newcommand*\lineheight[1]{\fontsize{\fsize}{#1\fsize}\selectfont}%
  \ifx\svgwidth\undefined%
    \setlength{\unitlength}{199.27298238bp}%
    \ifx\svgscale\undefined%
      \relax%
    \else%
      \setlength{\unitlength}{\unitlength * \real{\svgscale}}%
    \fi%
  \else%
    \setlength{\unitlength}{\svgwidth}%
  \fi%
  \global\let\svgwidth\undefined%
  \global\let\svgscale\undefined%
  \makeatother%
  \begin{picture}(1,0.74965599)%
    \lineheight{1}%
    \setlength\tabcolsep{0pt}%
    \put(0,0){\includegraphics[width=\unitlength,page=1]{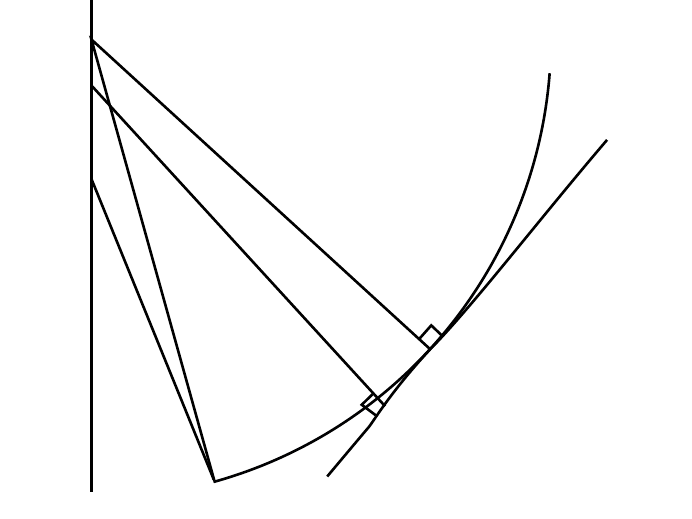}}%
    \put(0.00612993,0.68326662){\color[rgb]{0,0,0}\makebox(0,0)[lt]{\lineheight{1.25}\smash{\begin{tabular}[t]{l}$\Pi(\theta)$\end{tabular}}}}%
    \put(0.00612993,0.61797106){\color[rgb]{0,0,0}\makebox(0,0)[lt]{\lineheight{1.25}\smash{\begin{tabular}[t]{l}$\Pi(\alpha)$\end{tabular}}}}%
    \put(0.00533881,0.48137625){\color[rgb]{0,0,0}\makebox(0,0)[lt]{\lineheight{1.25}\smash{\begin{tabular}[t]{l}$\Pi(\theta')$\end{tabular}}}}%
    \put(0.87618433,0.5077036){\color[rgb]{0,0,0}\makebox(0,0)[lt]{\lineheight{1.25}\smash{\begin{tabular}[t]{l}$C$\end{tabular}}}}%
    \put(0.69610917,0.592806){\color[rgb]{0,0,0}\makebox(0,0)[lt]{\lineheight{1.25}\smash{\begin{tabular}[t]{l}$\mathcal O$\end{tabular}}}}%
    \put(0.63566642,0.20729844){\color[rgb]{0,0,0}\makebox(0,0)[lt]{\lineheight{1.25}\smash{\begin{tabular}[t]{l}$C(\theta)$\end{tabular}}}}%
    \put(0.57087248,0.13051339){\color[rgb]{0,0,0}\makebox(0,0)[lt]{\lineheight{1.25}\smash{\begin{tabular}[t]{l}$C(\alpha)$\end{tabular}}}}%
    \put(0.29978423,0.00900224){\color[rgb]{0,0,0}\makebox(0,0)[lt]{\lineheight{1.25}\smash{\begin{tabular}[t]{l}$C(\theta')$\end{tabular}}}}%
    \put(0.17230162,0.59543361){\color[rgb]{0,0,0}\makebox(0,0)[lt]{\lineheight{1.25}\smash{\begin{tabular}[t]{l}$q_\alpha$\end{tabular}}}}%
    \put(0.49189777,0.37556662){\color[rgb]{0,0,0}\makebox(0,0)[lt]{\lineheight{1.25}\smash{\begin{tabular}[t]{l}$r(\theta)$\end{tabular}}}}%
    \put(0.27595823,0.21109452){\color[rgb]{0,0,0}\makebox(0,0)[lt]{\lineheight{1.25}\smash{\begin{tabular}[t]{l}$r(\theta)$\end{tabular}}}}%
  \end{picture}%
\endgroup
\end{figure}
A similar contradiction is obtained if we assume 
there exists  $\theta'\in (\theta,\pi/n]$, with $
\Vert C(\te') - (0,\Pi(\te)) \Vert =r(\theta)
$.\\  We get the wanted inclusion.

\item We easily check that the convex hull $H(\theta)$ of the $n$ pieces of disks 
$$
G_n\left(\left\{ (0,\Pi(\theta))+r (\cos\varphi, \sin\varphi) |\ 0\le r\le r(\theta), \ -\frac\pi2-\frac{\pi}n\le \varphi\le -\frac\pi2+\frac{\pi}n\right\}\right)
$$
contains $\bar B((0,\Pi(\theta)), r(\theta))$ (check for instance that the curvature of its boundary is everywhere smaller than $1/r(\theta)$). But $H(\theta)\subset \bar D$ since $\bar D$ is left invariant by $G_n$ and convex. As a conclusion, $\bar B((0,\Pi(\theta)), r(\theta))\subset \bar D$, so $(0,\Pi(\theta))\in S$.
\end{itemize}

b) Finally we prove that $S\subset G_n(\{0\}\times [-y_0,0])$. For $\theta\in [0,\pi/n]$ and $r\in (0, r(\theta))$, consider the point $P=(0,\Pi(\theta))+ r\nu(\theta)$. We have to prove that it does not belong to $S$. 
Consider $\theta'\in[0,2\pi)$ such that $C(\theta')$ minimizes the distance between $ P$ and $C$.
First note that we must have $\theta'\in [0,\pi/n]$, otherwise the minimizing segment would cross an axis of symmetry, allowing to construct a shorter segment from $P$ to $C$.
Next let us show that necessarily $\theta'=\theta$. Indeed, otherwise, the lines $D_{\theta}$ and $D_{\theta'}$ would then intersect at $P$.
Assume for instance that $\theta<\theta'$, then we would get
 that $\Pi(\theta)>\Pi(\theta')$, which is forbidden.
 Finally, since $d(P,C(\theta))<r(\theta)\le 1/\rho(\theta)$, the distance to $C$ is not singular at $P$ and $P$ cannot belong to $S$. 
Using all symmetries, this proves  that the complementary of $G_n(\{0\}\times [-y_0,0])$ in $\bar D$ does not meet the cutlocus~$S$ of distance to~$C$.

(2) Assume that the skeleton of $C$ is $G_n(\{0\}\times [-y_0,0]).$ Then for all $\theta\in (0,\pi/n)$, we have $B((0,\Pi(\theta)), r(\theta))\subset D$. This implies that $r(\theta)\le 1/\rho(\theta)$.  Then  by~\eqref{r} we get 
$$
\Pi'(\theta)=\frac{\frac1{\rho(\theta)}-r(\theta)}{\sin(\theta)}\ge 0,
$$
so $\Pi$ is non-decreasing.
\end{proof}

\begin{proposition} \label{prop_S}
The set of curve $\mathcal{S}_n^\downarrow $ is stable under the stochastic curvature flow namely \eqref{eq_sto_CF}. It is also stable under the usual deterministic curvature flow.
\end{proposition}
\begin{proof}
Let $C_0 $ be a curve in $\mathcal{S}_n^\downarrow $, and $\rho_0$ the associated curvature function, by hypothesis $ \partial \rho_0 (\te) \le 0$ for $\te \in  [0,\frac{\pi}{n}]$.
Let $C_t $ be the solution of the stochastic curvature started at $C_0$, namely the solution of:
 \begin{equation*} \left \{  \begin{array}{lcl}
 d_t C(t,u) & = & (-\rho_t(C(t,u)) )\nu_{C(t,u)}dt + \sqrt{2} \nu_{C(t,u)} dB_t  \\
  C(0,u) &= & C_0(u). \\
 \end{array}  \right. \end{equation*}
 Using the parametrization by the angle $\te $ of the tangent vector and the horizontal axis as above we have, denoting $\rho(t,\theta)\df\rho_t(\theta)$, 
 \begin{equation*} \left \{  \begin{array}{lcl}
  d_t\rho(t,\theta) = \rho^2(t,\theta) (\partial^2_{\theta} \rho(t,\theta)) dt+ \rho^2(t,\theta)   \left( 3\rho(t,\theta)dt - \sqrt{2} dB_t  \right), \\
  \rho(0,\cdot) = \rho_0.
 \end{array}  \right. \end{equation*}
 (see \eqref{eq_rho} with $h$ replaced by 0).
 
 Using Lemma \ref{conserve_convex} we get that $\rho_t>0$ for $ t < \tau $ where $\tau  $ is any lifetime of the stochastic curvature flow. Using It\^o formula, we have for $ 0 \le t < \tau  $
\begin{equation}\label{central} d \frac{1}{\rho_t(\theta)} = (-\partial^2_{\theta} \rho_t(\theta)   - \rho_t(\theta)) dt + \sqrt{2}dB_t,
\end{equation}
 
  Computations similar to those of the proof of  Theorem \ref{th_cores} show that $ I_{c,\rho_t}(2\pi) = I_{s,\rho_t}(2\pi) = 0  .$
Recall $ S_0 $ is the reflection with respect to the vertical axis.
Using the uniqueness of the stochastic curvature flow, we have 
$$ S_0 (C_t(C_0)) = C_t(S_0 (C_0))=C_t(C_0). $$ \par Doing the same thing with the rotation $R_{2\pi/n}$,
it follows that $\rho_t$ satisfies Equation \eqref{sym}.
To get the result we only have to show that $\partial_{\te} \rho_t(\te) \le 0$ for $  0 \le t < \tau $ and $\te \in (0, \frac{\pi}{n})$.\par

Differentiating \eqref{central} in $\te$ we get:
$$d \left(\frac{\partial_{\te}\rho_t(\theta) }{\rho^2_t(\theta) }  \right) = \partial^2_{\theta} (\partial_{\te}\rho_t(\theta)) dt + \partial_{\te}\rho_t(\theta) dt. $$
Let $ \psi_t(\te) = \frac{\partial_{\te}\rho_t(\theta) }{\rho^2_t(\theta) }$, then $ \psi $ satisfies the following partial differential equation with stochastic coefficient and with lifetime $\tau $ (see Lemma \ref{conserve_convex}):
\begin{align*}
\partial_t \psi_t(\te)&=  \partial^2_{\theta} (\psi_t(\te) \rho^2_t(\theta) ) + \psi_t(\te)\rho^2_t(\theta)\\
&= \rho^2_t(\theta) \partial^2_{\theta} \psi_t(\te) + 4 \rho_t(\theta) (\partial_{\theta} \psi_t(\te)) (\partial_{\theta} \rho_t(\theta) ) +  \psi_t(\te)  \left( \partial^2_{\theta} \rho^2_t(\theta) +   \rho^2_t(\theta)   \right)
\end{align*} 
with initial condition $\psi_0(\te) = \frac{\partial_{\te}\rho_0(\theta) }{\rho^2_0\theta) }  $. By hypothesis $\psi_0(\te) \le 0 $. Note also by the conservation
of the symmetry that we have the boundary conditions $ \psi_t(0)= \psi_t(\frac{\pi}{n})= 0.$
To show that $ \partial_{\te}\rho_t(\te)\le 0 $ for all $ t< \tau$ we will argue by contradiction.
Suppose that there exists $ t^{*} < \tau $ and $ \te \in [0, \frac{\pi}{n}]$ such that $ \partial_{\te}\rho_{t^{*}}(\te) > 0 $  so $\psi_{t^{*}}(\te) > 0 $.
 Let  $$ \mu = -2  \left( \Vert \partial^2_{\theta} \rho^2_.(.)\Vert_{[0, t^{*}] \times [0, \frac{\pi}{2}]}  + \Vert \rho^2_.(.) \Vert_{[0, t^{*}] \times [0, \frac{\pi}{2}]}  \right) > - \infty ,$$
and $ W_t(\te) = e^{\mu t} \psi_t(\te) $,  which satisfies the following equation: 
\begin{equation}\label{max}
\begin{split}
&\partial_t W_t(\te) \\&= \rho^2_t(\theta) \partial^2_{\theta}W_t(\te) + 4 \rho_t(\theta)(\partial_{\theta} \rho_t(\theta) )(\partial_{\theta} W_t(\te)) + W_t(\theta)( \partial^2_{\theta} \rho^2_t(\te) + \rho^2_t(\te) + \mu ).
\end{split}
\end{equation}
\par
Define $\aaaa\df\sup_{\te \in [0, \frac{\pi}{n}]} W_{t^{*}}(\te) >0 $, \bq t_0&\df &\inf \lt\{ t\le t^{*},\,s.t. \, \sup_{\te \in [0, \frac{\pi}{n}]}W_{t}(\te) = \aaaa \rt\} \eq
and let $ \te^{*}$ be such that $W_{t_0}(\te^{*}) = \aaaa$. From boundary conditions we have  $ \te^{*} \in ] 0, \frac{\pi}{n} [$.
At $(t_0,\te^{*}) $ we have 
\begin{align*}
\partial_{t}W_t(\te^{*})_{\vert _{t_0}} \ge 0,   \quad \partial^2_{\theta}W_{t_0}(\te)_{\vert_{\te^{*}}}  \le 0,  \quad \partial_{\theta}W_{t_0}(\te)_{\vert_{\te^{*}}}  = 0.
\end{align*}
Using equation \eqref{max} we get the contradiction, since
$$ 0 \le \partial_{t}W_t(\te^{*})_{\vert _{t_0}}  \le   \aaaa \frac{\mu}{2} <0 .$$
With a similar proof, we get  the second part, namely the conservation of the class $\mathcal{S}_n^\downarrow $ under  the usual deterministic curvature flow.    
\end{proof}

\begin{corollary}
The class of domain $\mathcal{S}_n^\downarrow $ is also stable under the normalized stochastic curvature flow \eqref{eq_sto}.
\end{corollary}
\begin{proof}
Since the solutions of \eqref{eq_sto} are obtained by a change of probability from the solutions of the stochastic curvature flow, the state space does not change, and the result follows from Proposition \ref{prop_S}. 
\end{proof}
 
 \section{A new isoperimetric estimate} 
 
 Let us end our consideration of $\mathcal{S}_n$ by observing that its elements are quite round when $n^2$ is much larger that the length of their skeleton:
 \begin{proposition}\label{isoperimetric}
 For any curve $C$ in the set $\mathcal{S}_n $ defined in~\eqref{Sn} (and in particular with   skeleton $ G_n\left(\{0\}\times \left[ -{L(C)}/{n},0\right]\right) $)
we have
\begin{align*} \pi^2 (\rao -\rai)^2 &\le \sigma^2(C) - 4\pi \vol(D)\\&
   \le\ \frac{2\pi^2 }{n^2}  L(C)^2 \left(1-\frac{\sin\left(\frac{2\pi}{n}\right)}{\frac{2\pi}{n}}\right)\\&
\le \frac{4\pi^4}{3n^4}L(C)^2 ,
\end{align*}
 where  $ L(C)$ is the length of the skeleton of $C$.
 \end{proposition}
 \begin{proof}
The lower bound on $\sigma^2(C) - 4\pi \vol(D)$ is just Bonnesen inequality \eqref{Bonne}. For the upper bound, let
 $\rho$ be the  curvature function associated to $C$, and $p(\te)= \langle C(\te), \nu(\te) \rangle$ the support function.
Using computation in \eqref{r} we have 
$$ p(\te) =  -\Pi(\te)\cos(\te)  + r(\te) ,$$
$$ p'(\te) = \Pi(\te) \sin(\te)   $$ 
 $$ p''(\te) + p(\te) = \frac{1}{\rho(\te)}. $$
 By symmetry of $C$  we have the following Fourier series  of $p$:
 $$ p(\te) = a_0 + \sum_{k \ge 1} a_k \cos(kn\te) .$$ 
 Also $  \vol(D) = \frac12 \int_0^{2\pi} p(\te) (p(\te) + p''(\te)) d\te =  \pi a_0^2 + \frac{\pi}{2} \sum_{k \ge 2} a^2_k (1-n^2k^2  )$
 and $a_0 =\frac{1}{2\pi} \int  p(\te) d\te = \frac{1}{2\pi} \sigma (C).$ 
 Hence
 $$\begin{aligned} \sigma^2(C) - 4\pi \vol(D)  &= 2 \pi^2 \sum_{k\ge 1}a^2_k(n^2k^2 -1  )\\
 &\le 2 \pi \int_0^{2\pi} p'(\te)^2 d \te \\
 &= 4n\pi \int_0^{\pi/n} \Pi^2(\te) \sin^2(\te) d\te \\
 & \le 4n \pi \left( \frac{ L(C)}{n}\right)^2 \int_0^{\pi/n} \sin^2(\te) d\te \\
 &= \frac{2\pi^2 }{n^2}  L(C)^2 \left(1-\frac{\sin\left(\frac{2\pi}{n}\right)}{\frac{2\pi}{n}}\right)\\
&\le  \frac{4\pi^4}{3n^4}L(C)^2
 \end{aligned}$$
since $1-\sin(x)/x\leq x^2/6$ for any $x\in \RR$ (with the usual convention $\sin(0)/0=1$).
  \end{proof}
  \begin{figure}[h]
\begingroup%
  \makeatletter%
  \providecommand\color[2][]{%
    \errmessage{(Inkscape) Color is used for the text in Inkscape, but the package 'color.sty' is not loaded}%
    \renewcommand\color[2][]{}%
  }%
  \providecommand\transparent[1]{%
    \errmessage{(Inkscape) Transparency is used (non-zero) for the text in Inkscape, but the package 'transparent.sty' is not loaded}%
    \renewcommand\transparent[1]{}%
  }%
  \providecommand\rotatebox[2]{#2}%
  \newcommand*\fsize{\dimexpr\f@size pt\relax}%
  \newcommand*\lineheight[1]{\fontsize{\fsize}{#1\fsize}\selectfont}%
  \ifx\svgwidth\undefined%
    \setlength{\unitlength}{213.63245927bp}%
    \ifx\svgscale\undefined%
      \relax%
    \else%
      \setlength{\unitlength}{\unitlength * \real{\svgscale}}%
    \fi%
  \else%
    \setlength{\unitlength}{\svgwidth}%
  \fi%
  \global\let\svgwidth\undefined%
  \global\let\svgscale\undefined%
  \makeatother%
  \begin{picture}(1,0.84698057)%
    \lineheight{1}%
    \setlength\tabcolsep{0pt}%
    \put(0,0){\includegraphics[width=\unitlength,page=1]{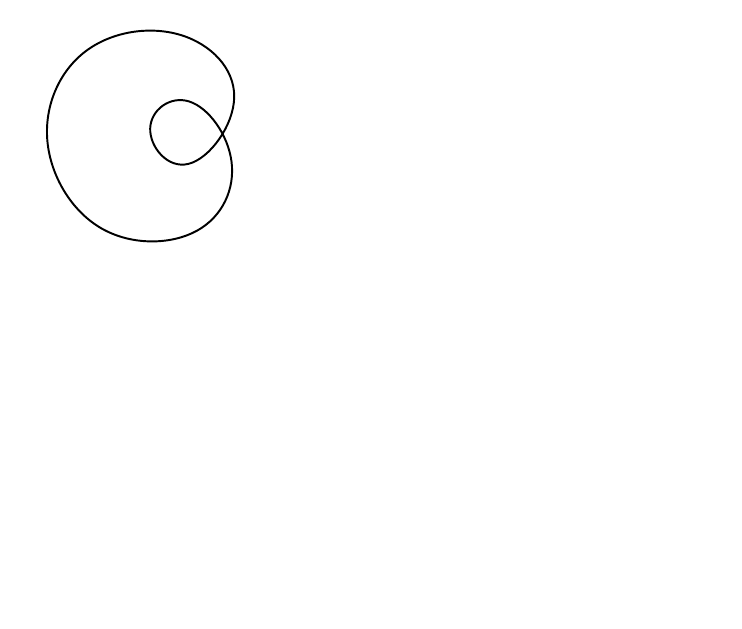}}%
    \put(-0.00333905,0.79756283){\color[rgb]{0,0,0}\makebox(0,0)[lt]{\lineheight{1.25}\smash{\begin{tabular}[t]{l}a)\end{tabular}}}}%
    \put(-0.00591487,0.35265417){\color[rgb]{0,0,0}\makebox(0,0)[lt]{\lineheight{1.25}\smash{\begin{tabular}[t]{l}b)\end{tabular}}}}%
    \put(0,0){\includegraphics[width=\unitlength,page=2]{dessin1.pdf}}%
    \put(0.24061069,0.01535956){\color[rgb]{0,0,0}\makebox(0,0)[lt]{\lineheight{1.25}\smash{\begin{tabular}[t]{l}$t=0$\end{tabular}}}}%
    \put(0,0){\includegraphics[width=\unitlength,page=3]{dessin1.pdf}}%
    \put(0.7342763,0.01535956){\color[rgb]{0,0,0}\makebox(0,0)[lt]{\lineheight{1.25}\smash{\begin{tabular}[t]{l}$t=T_c$\end{tabular}}}}%
  \end{picture}%
\endgroup
\caption{\label{dessin}}
\end{figure}
  

\end{document}